 \documentclass[a4paper,reqno,final,11pt]{amsart}
\usepackage{amssymb}
\usepackage{amsmath}
\usepackage{amsfonts}
\usepackage{hyperref}
\usepackage{graphicx}

\newtheorem{theorem}{Theorem}
\theoremstyle{plain}
\newtheorem{corollary}{Corollary}

\newtheorem{lemma}{Lemma}

\newtheorem{remark}{Remark}
\numberwithin{equation}{section}

 \title[Observability and Controllability of the Wave Equation]{Exact Boundary
Observability and Controllability of the Wave Equation in an Interval
with two Moving Endpoints}
\author[A. Sengouga]{Abdelmouhcene Sengouga}
\address[Abdelmouhcene Sengouga]{ Laboratory of Functional Analysis and
Geometry of Spaces\\
Department of mathematics\\
Faculty of Mathematics and Computer Sciences\\
University of M'sila\\
28000 M'sila, Algeria.}
\email{amsengouga@gmail.com}
\date{\today}
\begin{document}

\begin{abstract}begin{abstract}
We study the wave equation in an interval with two linearly moving
endpoints. We give the exact solution by a series formula, then we show that
the energy of the solution decay at the rate $1/t$. We also establish
observability results, at one or two endpoints,\ in a sharp time. Moreover,
using the Hilbert uniqueness method, we derive exact boundary
controllability results.
\end{abstract}
\maketitle

\section{Introduction}

The wave equation is a simple mathematical model describing the transverse
small vibrations of a homogeneous string under tension and constrained to
move in a plane. Let $T>0.$ When the string endpoints are clamped and its
length $\ell _{0}>0\ $is invariant with time, this model can be stated as%
\begin{equation*}
\left\{ 
\begin{array}{ll}
w_{tt}-w_{xx}=0,\ \smallskip  & \text{for }\left( x,t\right) \in \left(
0,\ell _{0}\right) \times \left( 0,T\right) , \\ 
w\left( 0,t\right) =w\left( \ell ,t\right) =0,\smallskip  & \text{for \ }%
t\in \left( 0,T\right) , \\ 
w(x,0)=w^{0},\ \ w_{t}\left( x,0\right) =w^{1}\text{, \ } & \text{for \ \ }%
x\in \left( 0,\ell _{0}\right) ,%
\end{array}%
\right. 
\end{equation*}%
see \cite{Str:08}. The function $x\rightarrow w\left( x,t\right) ,x\in
\left( 0,\ell _{0}\right) $ describes the shape of the string at time $t.$
It is well known that for initial data satisfying $w^{0}\in H_{0}^{1}\left(
0,\ell _{0}\right) ,w^{1}\in L^{2}\left( 0,\ell _{0}\right) ,$ the solution
of this problem is unique and enjoys, in particular, the following
properties :

\begin{itemize}
\item The \textquotedblleft energy\textquotedblright\ of the solution, given
by 
\begin{equation*}
E_{0}\left( t\right) =\frac{1}{2}\int_{0}^{\ell
_{0}}w_{x}^{2}(x,t)+w_{t}^{2}(x,t)\text{ }dx,\ \ \ \ \text{ }t\in \left(
0,T\right) ,
\end{equation*}%
is a conserved quantity, i.e. $E_{0}\left( t\right) =Constante$, $\forall
t\geq 0$.

\item Due to the finite speed of propagation (here equal to $1$), the
observability and the controllability at one endpoint hold$\ $if and only if
the length of the time interval satisfies $T\geq 2\ell _{0}.$

\item The observability and the controllability at the two endpoints holds
if and only if $T\geq \ell _{0}.$
\end{itemize}

If the length of the string varies in time, we wonder if the solution of the
wave equation has some analogue proprieties. Such situation where the
spacial domain is time-dependent appears in many different areas of physics,
from optics, electromagnetism, fluid dynamics to quantum mechanics. See for
instance \cite{DKN:93, MLM:02, Moo:70,Wan:90} and the survey paper \cite%
{KnK:14}.

To be more precise, we consider the wave equation in an interval with two
linearly moving endpoints. We suppose that the left endpoint moves to the
left with a constant speed $\ell _{1}$ and the other endpoint moves to the
right at a constant speed $\ell _{2}$. We assume that%
\begin{equation}
0\leq \ell _{1},\ell _{2}<1\text{ \ \ and \ \ }\ell _{1}+\ell _{2}>0.
\label{tlike}
\end{equation}%
The Later condition ensures that the length of the string is increasing. Let
us denote by $L_{0}$ its initial length. In order to simplify notation and
also to make the computations easier, we take as an initial time 
\begin{equation*}
t_{0}=L_{0}/\left( \ell _{1}+\ell _{2}\right) 
\end{equation*}%
and consider the following interval with moving ends%
\begin{equation*}
I_{t}:=\left( -\ell _{1}t,\ell _{2}t\right) ,\text{ \ }t\geq t_{0}.
\end{equation*}%
In the $xt-$plan, we have a noncylindrical domain $Q_{t_{0}+T},$ and its
lateral boundary $\Sigma _{t_{0}+T},$ defined as%
\begin{gather*}
Q_{t_{0}+T}:=\left\{ \left( x,t\right) \in 
\mathbb{R}
^{2}\text{ }|\text{ }-\ell _{1}t<x<\ell _{2}t,\text{ \ \ for }t\in \left(
t_{0},t_{0}+T\right) \right\} , \\
\Sigma _{t_{0}+T}:=\bigcup_{t_{0}<t<t_{0}+T}\left\{ \left( -\ell
_{1}t,t\right) ,\left( \ell _{2}t,t\right) \right\} .
\end{gather*}%
The assumption $0\leq \ell _{1},\ell _{2}<1$ in  (\ref{tlike}) ensures that $%
\Sigma _{t_{0}+T}$ satisfies the so-called timelike condition.

Let us now consider the wave equation, with homogeneous Dirichlet boundary
conditions, 
\begin{equation}
\left\{ 
\begin{array}{ll}
\phi _{tt}-\phi _{xx}=0,\ \smallskip  & \text{in }Q_{t_{0}+T}, \\ 
\phi \left( -\ell _{1}t,t\right) =\phi \left( \ell _{2}t,t\right)
=0,\smallskip  & \text{for \ }t\in \left( t_{0},t_{0}+T\right) , \\ 
\phi (x,t_{0})=\phi ^{0}\left( x\right) \text{, \ \ }\phi _{t}\left(
x,t_{0}\right) =\phi ^{1}\left( x\right) ,\text{ \ } & \text{for \ \ }x\in
I_{t_{0}}.%
\end{array}%
\right.   \tag{$WP$}  \label{wave}
\end{equation}%
Under the assumption (\ref{tlike}) and for every initial data%
\begin{equation}
\phi ^{0}\in H_{0}^{1}\left( I_{t_{0}}\right) ,\text{ \ }\phi ^{1}\in
L^{2}\left( I_{t_{0}}\right) \   \label{ic}
\end{equation}%
there exists a unique solution to Problem (\ref{wave}) such that%
\begin{equation}
\phi \in C\left( [t_{0},t_{0}+T];H_{0}^{1}\left( I_{t}\right) \right) ,\text{
\ \ }\phi _{t}\in C\left( [t_{0},t_{0}+T];L^{2}\left( I_{t}\right) \right) .
\label{solreg}
\end{equation}%
see \cite{BaC:81,CoB:73,Lion:69}. We define the \textquotedblleft
energy\textquotedblright\ of the solution of Problem (\ref{wave}) as%
\begin{equation}
E\left( t\right) =\frac{1}{2}\int_{-\ell _{1}t}^{\ell _{2}t}\phi
_{x}^{2}(x,t)+\phi _{t}^{2}(x,t)\text{ }dx,\ \ \ \ \ \text{for }t\geq t_{0}.
\label{E}
\end{equation}%
As we will see below, $E\left( t\right) $ is decaying in time. This
contrasts the conservation of the energy $E_{0}\left( t\right) $ defined
above.

Next, we consider the observability problem for (\ref{wave}) at an endpoint $%
\xi t$ for $\xi \in \left\{ -\ell _{1},\ell _{2}\right\} .$ This problem can
be stated as follows: to give sufficient conditions on the length of the
time interval, denoted by $T_{\ell },$ such that there exists a constant $%
C(T_{\ell })>0$ for which the observability inequality%
\begin{equation*}
E\left( t_{0}\right) \leq C(T_{\ell })\int_{t_{0}}^{t_{0}+T_{\ell }}\phi
_{x}^{2}\left( \xi t,t\right) +\phi _{t}^{2}\left( \xi t,t\right) dt,
\end{equation*}%
holds for all the solutions of (\ref{wave}). This inequality is also called
the inverse inequality. It allows estimating the energy of solutions in
terms of the energy localized at the boundary $x=\xi t$. Observe that $\phi
\left( \xi t,t\right) =0\ $yields $\left( \phi \left( \xi t,t\right) \right)
_{t}=\xi \phi _{x}\left( \xi t,t\right) +\phi _{t}\left( \xi t,t\right) =0.$
Then denoting $C_{\xi }=C(T_{\ell })\left( 1-\xi ^{2}\right) $, the
precedent inequality can be rewritten as%
\begin{equation}
E\left( t_{0}\right) \leq C_{\xi }(T_{\ell })\int_{t_{0}}^{t_{0}+T_{\ell
}}\phi _{x}^{2}\left( \xi t,t\right) \text{ }dt.  \label{ct}
\end{equation}%
The minimal value of $T_{\ell },$ for which Inequality (\ref{ct}) holds, is
called the time of observability. Due to the finite speed of propagation,
one expects that $T_{\ell }>0$ depends on the initial length $L_{0}$ and\
also on the two speeds of expansion $\ell _{1}$ and $\ell _{2}$. This
dependence is simply denoted by subscript $\ell $ in the notation of $%
T_{\ell }$.

On the other hand, we consider the following boundary controllability
problem: given 
\begin{eqnarray}
(u^{0},u^{1}) &\in &L^{2}\left( I_{t_{0}}\right) \times H^{-1}\left(
I_{t_{0}}\right) ,  \label{u0} \\
(u_{T}^{0},u_{T}^{1}) &\in &L^{2}\left( I_{t_{0}+T}\right) \times
H^{-1}\left( I_{t_{0}+T}\right) ,  \label{uT}
\end{eqnarray}%
find a control function $v\in L^{2}\left( t_{0},t_{0}+T\right) ,$ acting at
one of the endpoints, say $x=\ell _{2}t,$ such that the solution of the
problem%
\begin{equation}
\left\{ 
\begin{array}{ll}
u_{tt}-u_{xx}=0,\ \smallskip  & \text{in }Q_{t_{0}+T}, \\ 
u\left( -\ell _{1}t,t\right) =0,\text{ \ \ }u\left( \ell _{2}t,t\right)
=v\left( t\right) ,\smallskip  & \text{for \ \ }t\in \left(
t_{0},t_{0}+T\right) , \\ 
u(x,t_{0})=u^{0}\left( x\right) ,\text{ \ \ }u_{t}\left( x,t_{0}\right)
=u^{1}\left( x\right) ,\text{ \ } & \text{for \ \ }x\in I_{t_{0}},%
\end{array}%
\right.   \tag{$CWP$}  \label{wavec}
\end{equation}%
satisfies also 
\begin{equation*}
u(x,t_{0}+T)=u_{T}^{0}\left( x\right) ,\text{ \ \ }u_{t}\left(
x,t_{0}+T\right) =u_{T}^{1}\left( x\right) ,\text{ \ \ for \ \ }x\in
I_{t_{0}+T}.
\end{equation*}%
Note that Problem (\ref{wavec}) admits a unique solution%
\begin{equation*}
u\in C([t_{0},t_{0}+T];L^{2}\left( I_{t}\right) )\cap
C^{1}([t_{0},t_{0}+T];H^{-1}\left( I_{t}\right) ).
\end{equation*}%
in the transposition sense, see \cite{Mir:96}.

We shall also pay attention to the problem of observability at both ends
where inequality (\ref{ct}) is replaced by 
\begin{equation}
E\left( t_{0}\right) \leq C(\tilde{T}_{\ell })\int_{t_{0}}^{t_{0}+\tilde{T}%
_{\ell }}\phi _{x}^{2}\left( -\ell _{1}t,t\right) +\phi _{x}^{2}\left( \ell
_{2}t,t\right) dt,  \label{ct2}
\end{equation}%
for some time $\tilde{T}_{\ell }$ and a constant $C(\tilde{T}_{\ell }).$
Then we consider the associated controllability problem at both ends. That
is to say, for any 
\begin{eqnarray}
(y^{0},y^{1}) &\in &L^{2}\left( I_{t_{0}}\right) \times H^{-1}\left(
I_{t_{0}}\right) ,  \label{y0} \\
(y_{T}^{0},y_{T}^{1}) &\in &L^{2}\left( I_{t_{0}+T}\right) \times
H^{-1}\left( I_{t_{0}+T}\right) ,  \label{yT}
\end{eqnarray}%
find two control functions $v_{1},v_{2}\in L^{2}\left( t_{0},t_{0}+T\right) $
such that the solution of 
\begin{equation}
\left\{ 
\begin{array}{ll}
y_{tt}-y_{xx}=0,\ \smallskip  & \text{in }Q_{t_{0}+T}, \\ 
y\left( -\ell _{1}t,t\right) =v_{1}\left( t\right) ,\text{ \ \ }y\left( \ell
_{2}t,t\right) =v_{2}\left( t\right) ,\smallskip  & \text{for \ \ }t\in
\left( t_{0},t_{0}+T\right) , \\ 
y(x,t_{0})=y^{0}\left( x\right) ,\text{ \ \ }y_{t}\left( x,t_{0}\right)
=y^{1}\left( x\right) ,\text{ \ \ \ \ \ } & \text{for \ \ }x\in I_{t_{0}},%
\end{array}%
\right.   \tag{$CWP2$}  \label{wavec2}
\end{equation}%
satisfies the final conditions 
\begin{equation*}
y(x,t_{0}+T)=y_{T}^{0}\left( x\right) ,\text{ \ \ }y_{t}\left(
x,t_{0}+T\right) =y_{T}^{1}\left( x\right) ,\text{ \ \ for \ \ }x\in
I_{t_{0}+T}.
\end{equation*}

Observability and controllability of the wave equation in noncylindrical
domains were considered by several authors. Bardos and Chen \cite{BaC:81}
obtained the interior exact controllability of the wave equation, in
noncylindrical domains, by a "controllability via stabilisation" argument.
Miranda \cite{Mir:96} used a change of variable to transform the
noncylindrical problem to a cylindrical one, shows the exact boundary
controllability by the Hilbert Uniqueness method (see \cite{Lion:88}) then,
going back to the noncylindrical problem, he obtain the desired results. In
recent years, there is a renewed interest in the observability and
controllability\ of such problems, see for instance \cite%
{CLG:13,LLC:15,SLL:15}.

The authors, in these cited works, relay on the multiplier method (see \cite%
{Komo:94}) to establish the energy estimates and inequalities necessary to
derive the observability and controllability results. In this work, we
present a different approach. It was inspired by the work of Balazs \cite%
{Bal:61} where he obtained the exact solution of the 1--d wave equation as
the sum of a generalised Fourier series, i.e., a countable set of orthogonal
functions in a weighted $L^{2}-$spaces. The key idea is that we analyse the
series representation of solution and use it to derive the desired energy
and observability estimates. Then using HUM, we establish the
controllability of the wave equation.

Fourier series approach in control theory of problems in cylindrical domains
is by now classic, see \cite{KoLo:05,Rus:78,Zua:05}. However, the use of
this approach for problems in noncylindrical problems seems to be new.
Recently we have successfully applied this approach \cite{Sen:18,Sen:18b} to the
wave equation in an interval with one moving endpoint $\left( 0,\ell
_{0}t\right) ,0<\ell _{0}<1.$ We showed that the boundary observability and
controllability at one endpoint, whether it is the fixed or the moving one,
holds in a sharp time $T_{0}=2L_{0}/\left( 1-\ell _{0}\right) $. In the
paper at hand, we consider a wave equation in an interval with two moving
endpoints. We show observability results at one endpoint and also at the two
endpoints. Once the observability is established, the controllability is
derived by HUM.

The main contribution of this work can be summarized as follows:

\begin{itemize}
\item The series formulas of the exact solution of (\ref{wave}) was given by
Balazs \cite{Bal:61}. He managed to calculate the coefficients of the series
when only one endpoint is moving, e.g. $\ell _{1}=0.$ If the two endpoints
are moving, i.e. $\ell _{1}>0$ and $\ell _{2}>0,$ the question remained open
as far as we know. Here we show explicitly how to calculate these
coefficients in function of the initial data $\phi ^{0}$and\ $\phi ^{1},$
see Theorem \ref{thexist1}.

\item The decay of the energy of the solution of (\ref{wave}) is not
conserved and decay at the precise rate $1/t,$ see Theorem \ref{th1}.

\item The observability of (\ref{wave}) and the controllability of (\ref%
{wavec}), at an endpoint, holds if and only if 
\begin{equation}
T\geq T_{\ell }:=\frac{2L_{0}}{\left( 1-\ell _{1}\right) \left( 1-\ell
_{2}\right) }.  \label{T1}
\end{equation}%
This value of $T_{\ell }$ holds whether the endpoint is $x=-\ell _{1}t$ or $%
x=\ell _{2}t,$ see Theorems \ref{thobs1} and \ref{thc1}.

\item Problem (\ref{wave}) is observable and Problem (\ref{wavec2}) is
controllable, at both endpoints, if and only if 
\begin{equation}
T\geq \tilde{T}_{\ell }:=\max \left\{ \frac{L_{0}}{\left( 1-\ell _{1}\right) 
},\frac{L_{0}}{\left( 1-\ell _{2}\right) }\right\} ,  \label{T2}
\end{equation}%
see Theorems \ref{thobs2} and \ref{thc2}.

\item The values of $T_{\ell },\tilde{T}_{\ell }$ and all the obtained
constants, in the direct and inverse inequalities, do not depend on the
choice of the initial time $t_{0}$.
\end{itemize}

These results are new to our knowledge. In particular, the observability and
the controllability at both endpoints (\ref{wavec2}) seem to be not
considered before. As mentioned in \cite{Sen:18,Sen:18b}, when an endpoint is a
fixed one, e.g. $\ell _{1}=0,$ the obtained time of observability is sharp
and improves other results obtained by the multiplier method, see for
instance \cite{CLG:13,CJW:15,SLL:15}.

The remainder of this paper is organized as follows. For the convenience of
the reader, we recall in section 2 some definitions and establish some
lemmas that are necessary for the sequel. In section 3, we show how to
calculate the coefficients of the series formula that gives the solution of (%
\ref{wave}). In section 4, we derive a sharp estimate for the energy of the
solution of (\ref{wave}). Then, the boundary observability and
controllability at one and at both endpoints are considered in the fifth and
sixth sections.

\section{Preliminaries}

First, let us fix the notation of some constants that will frequently appear
in the sequel.%
\begin{gather*}
l:=\min \left\{ \ell _{1},\ell _{2}\right\} ,\text{ \ \ \ }L:=\max \left\{
\ell _{1},\ell _{2}\right\} , \\
\alpha _{\ell }:=\frac{1+\ell _{1}}{1-\ell _{2}},\text{ \ \ }\beta _{\ell }:=%
\frac{1+\ell _{2}}{1-\ell _{1}},\text{ \ \ }\kappa _{\ell }:=\frac{2}{\log
\left( \alpha _{\ell }\beta _{\ell }\right) } \\
L_{1}:=\left( 1-\ell _{2}\right) \alpha _{\ell }\beta _{\ell }-1,\text{ \ \
\ \ }L_{2}:=\left( 1-\ell _{1}\right) \alpha _{\ell }\beta _{\ell }-1.
\end{gather*}%
In these notation, $T_{\ell }$ and $\tilde{T}_{\ell }$, defined in (\ref{T1}%
) and\ (\ref{T2}), can be expressed as 
\begin{equation*}
T_{\ell }=\left( \alpha _{\ell }\beta _{\ell }-1\right) t_{0}\ \ \text{and }%
\tilde{T}_{\ell }=\left( \max \left\{ \alpha _{\ell },\beta _{\ell }\right\}
-1\right) t_{0}.
\end{equation*}%
Taking into account Assumption (\ref{tlike}), we can also check that 
\begin{equation*}
\alpha _{\ell }>1,\text{ }\beta _{\ell }>1,\text{ }0<\kappa _{\ell }<+\infty 
\text{ \ and \ }-L_{1}<-\left( \ell _{1}+\ell _{2}\right) <-\ell _{1}<\ell
_{2}<\ell _{1}+\ell _{2}<L_{2}.
\end{equation*}

Let $a,b\in 
\mathbb{R}
$ such that $b>a,$ and consider a positive (weight) function $\rho :\left(
a,b\right) \rightarrow 
\mathbb{R}
$. In the sequel, we denote by $L^{2}\left( a,b,\rho ds\right) $ the
weighted Hilbert space of measurable complex valued functions on $%
\mathbb{R}
,$ endowed by the scalar product%
\begin{equation*}
\int\limits_{a}^{b}f\left( s\right) \overline{g\left( s\right) }\rho \left(
s\right) ds
\end{equation*}%
and its associated norm. As usual, we drop $\rho ds$ in the $L^{2}$ space
notation if $\rho =1$.

If the set of functions $\left\{ \varphi _{n}\right\} _{n\in 
\mathbb{Z}
}$ is a complete orthonormal basis of $L^{2}\left( a,b,\rho ds\right) ,$
then every function $f\in L^{2}\left( a,b,\rho ds\right) $ can be written as 
\begin{equation}
f\left( s\right) =\sum_{n\in 
\mathbb{Z}
}a_{n}\varphi _{n}\left( s\right) ,\text{ \ where \ }a_{n}:=\int_{a}^{b}f%
\left( s\right) \overline{\varphi _{n}\left( s\right) }\rho (s)ds.
\end{equation}%
In particular, the following Parseval equality holds%
\begin{equation}
\int_{a}^{b}\left\vert f(s)\right\vert ^{2}\rho (s)ds=\sum_{n\in 
\mathbb{Z}
}\left\vert a_{n}\right\vert ^{2},  \label{prsvl}
\end{equation}%
see for instance \cite{BiR:89,Pin:09}.

Let us check the completeness of some sets of functions used in the next
sections.

\begin{lemma}
\label{lm01}Let $M$ be a positive integer and let $a,b\in 
\mathbb{R}
^{\ast }$ satisfying $\frac{b}{a}=\left( \alpha _{\ell }\beta _{\ell
}\right) ^{M}.$ Then, the set of functions 
\begin{equation}
\left\{ \sqrt{\kappa _{\ell }/2M}\ e^{i\pi n\left( \kappa _{\ell }/M\right)
\log z}\right\} _{n\in 
\mathbb{Z}
}  \label{base}
\end{equation}%
is complete and orthonormal in the space $L^{2}\left( a,b,dz/z\right) .$
\end{lemma}

\begin{proof}
The orthonormality follows easily. The completeness holds if and only if no
nonzero function in $L^{2}\left( a,b,dz/z\right) $ is orthogonal to all the $%
e^{i\pi n\left( \kappa _{\ell }/M\right) \log z}$, $n\in 
\mathbb{Z}
$. Indeed, let $f\in L^{2}\left( a,b,dz/z\right) $ and consider the change
of variable%
\begin{equation*}
s=\frac{\kappa _{\ell }}{M}\log \frac{z}{a},\text{ \ \ \ \ \ }a\leq z\leq b,
\end{equation*}%
then we have%
\begin{equation*}
\int\limits_{a}^{b}f\left( z\right) e^{-i\pi n\left( \kappa _{\ell
}/M\right) \log z}\text{ }\frac{dz}{z}=0\Leftrightarrow
\int\limits_{0}^{2}f\left( ae^{Ms/\kappa _{\ell }}\right) e^{i\pi ns}\text{ }%
ds=0.
\end{equation*}%
for every $n\in 
\mathbb{Z}
.$ A well-known result in analysis is that the set of functions $\left\{
e^{i\pi ns}\right\} _{n\in 
\mathbb{Z}
}$ is complete and orthogonal in $L^{2}\left( 0,2\right) .$ Thus $%
f(ae^{Ms/\kappa _{\ell }})=0,$ for a.e. $s\in \left( 0,2\right) $, i.e.%
\begin{equation*}
f\left( z\right) =0,\text{ \ for a.e. }z\in \left( a,b\right) .
\end{equation*}%
This ends the proof.
\end{proof}

\begin{remark}
\label{rmk00}Taking $a=t_{0},b=\left( \alpha _{\ell }\beta _{\ell }\right)
^{M}t_{0},$ then the set \emph{(\ref{base})} is a complete orthonormal set
in the space $L^{2}\left( t_{0},\left( \alpha _{\ell }\beta _{\ell }\right)
^{M}t_{0},dt/t\right) .$
\end{remark}

\begin{remark}
Since $b>a>0,$ then $0<\frac{1}{b}\leq \frac{1}{z}\leq \frac{1}{a},$ for $%
z\in \left( a,b\right) $ and 
\begin{equation*}
\frac{1}{b}\left\Vert \text{\thinspace }\cdot \text{\thinspace }\right\Vert
_{L^{2}\left( a,b\right) }\leq \left\Vert \text{\thinspace }\cdot \text{%
\thinspace }\right\Vert _{L^{2}\left( a,b,dz/z\right) }\leq \frac{1}{a}%
\left\Vert \text{\thinspace }\cdot \text{\thinspace }\right\Vert
_{L^{2}\left( a,b\right) }.
\end{equation*}%
This means that $f\in L^{2}\left( a,b\right) $ if and only if $f\in
L^{2}\left( a,b,dz/z\right) .$
\end{remark}

\begin{lemma}
\label{lm02}For every $t\geq t_{0},$ the set of functions 
\begin{equation}
\left\{ \sqrt{\kappa _{\ell }/2}\ e^{in\pi \kappa _{\ell }\log \left(
t+x\right) }\right\} _{n\in 
\mathbb{Z}
}\text{, \ \ \ }\left( \text{resp. }\left\{ \sqrt{\kappa _{\ell }/2}\
e^{in\pi \kappa _{\ell }\log \left( t-x\right) }\right\} _{n\in 
\mathbb{Z}
}\right)   \label{base+}
\end{equation}%
is complete and orthonormal in the space $$L^{2}\left( -\ell _{1}t,L_{2}t,%
\frac{dx}{t+x}\right) , \ \left( \text{resp. }L^{2}\left( -L_{1}t,\ell
_{2}t,\frac{dx}{t-x}\right) \right). $$
\end{lemma}

\begin{proof}
For every $t\geq t_{0},$ we use the change of variable%
\begin{equation*}
s=\kappa _{\ell }\log \frac{t+x}{t\left( 1-\ell _{1}\right) },\ \ x\in
\left( -\ell _{1}t,L_{2}t\right) ,\text{ \ \ }\left( \text{resp. }s=2-\kappa
_{\ell }\log \frac{t-x}{t\left( 1-\ell _{2}\right) }\right) ,\ \ x\in \left(
-L_{1}t,\ell _{2}t\right) ,
\end{equation*}%
to obtain $s\in \left( 0,2\right) $. The result follows for by arguing as in
the proof of Lemma \ref{lm01}.
\end{proof}

\begin{remark}
\label{rmk2}Note that \emph{(\ref{tlike})} ensures that the weight functions 
$1/\left( t\pm x\right) $ are positive, hence%
\begin{eqnarray*}
\frac{1}{t\left( 1+L_{2}\right) }\left\Vert \text{\thinspace }\cdot \text{%
\thinspace }\right\Vert _{L^{2}\left( -\ell _{1}t,L_{2}t\right) } &\leq
&\left\Vert \text{\thinspace }\cdot \text{\thinspace }\right\Vert
_{L^{2}\left( -\ell _{1}t,L_{2}t,\frac{dx}{t+x}\right) }\leq \frac{1}{%
t\left( 1-\ell _{1}\right) }\left\Vert \text{\thinspace }\cdot \text{%
\thinspace }\right\Vert _{L^{2}\left( -\ell _{1}t,L_{2}t\right) },\text{ \ \ 
}t\geq t_{0},\smallskip \\
\frac{1}{t\left( 1+L_{1}\right) }\left\Vert \text{\thinspace }\cdot \text{%
\thinspace }\right\Vert _{L^{2}\left( -L_{1}t,\ell _{2}t\right) } &\leq
&\left\Vert \text{\thinspace }\cdot \text{\thinspace }\right\Vert
_{L^{2}\left( -L_{1}t,\ell _{2}t,\frac{dx}{t-x}\right) }\leq \frac{1}{%
t\left( 1-\ell _{2}\right) }\left\Vert \text{\thinspace }\cdot \text{%
\thinspace }\right\Vert _{L^{2}\left( -L_{1}t,\ell _{2}t\right) },\text{ \ \ 
}t\geq t_{0}
\end{eqnarray*}%
and therefore $f$ belongs to $L^{2}\left( -\ell _{1}t,L_{2}t,\frac{dx}{t+x}%
\right) $ or $L^{2}\left( -L_{1}t,\ell _{2}t,\frac{dx}{t-x}\right) $ if and
only if $f\in L^{2}\left( -\ell _{1}t,L_{2}t\right) $ or $f\in L^{2}\left(
-L_{1}t,\ell _{2}t\right) $ respectively.
\end{remark}

\section{Exact solution}

According to Balazs \cite{Bal:61}, the exact solution of Problem (\ref{wave}%
) is$\ $given by the series%
\begin{equation}
\phi (x,t)=\sum_{n\in 
\mathbb{Z}
^{\ast }}c_{n}\left( e^{in\pi \kappa _{\ell }\log \left( t+x\right)
}-e^{in\pi \kappa _{\ell }\log \left( \frac{1+\ell _{2}}{1-\ell _{2}}\right)
}e^{in\pi \kappa _{\ell }\log \left( t-x\right) }\right) ,\ \ x\in \left(
-\ell _{1}t,\ell _{2}t\right) ,t\geq t_{0},  \label{exact0}
\end{equation}%
where $c_{n}$ are complex numbers, independent of $t$. Unfortunately, he
showed how to calculate the coefficients $c_{n}$, in function of the initial
data, only when the interval has one moving endpoint$.$ If the two endpoints
are moving, i.e. $\ell _{1}>0$ and $\ell _{2}>0$, the determination of $%
c_{n} $ turns out to be a bit tricky as we will see below.

First, we note that the two orthogonal sets considered in (\ref{base+})
appear in the series formula (\ref{exact0}). To use their orthogonality
properties, we need to extend the function $\phi ,$ defined only on $\left(
-\ell _{1}t,\ell _{2}t\right) ,$ to the intervals $\left( -L_{1}t,\ell
_{2}t\right) $ and $\left( -\ell _{1}t,L_{2}t\right) $, (recall that $%
-L_{1}<-\ell _{1}<\ell _{2}<L_{2}$). This extension is realised as follows%
\begin{equation}
\tilde{\phi}(x,t)=\left\{ 
\begin{array}{ll}
-\phi \left( -t+\frac{1-\ell _{1}}{1+\ell _{1}}\left( t-x\right) ,t\right) 
& x\in \left( -L_{1}t,-\ell _{1}t\right) ,\smallskip  \\ 
\phi \left( x,t\right)  & x\in \left( -\ell _{1}t,\ell _{2}t\right)
,\smallskip  \\ 
-\phi \left( t-\frac{1-\ell _{2}}{1+\ell _{2}}\left( t+x\right) ,t\right) 
\text{ \ \ } & x\in \left( \ell _{2}t,L_{2}t\right) .%
\end{array}%
\right.   \label{phi0+}
\end{equation}%
The obtained function is well defined since the first variable of $\phi $
remain in the interval $\left( -\ell _{1}t,\ell _{2}t\right) .$ This
extension ensures some odd-like symmetry in the variable $x$ for $\tilde{\phi%
}$ with respect to the point $x=-\ell _{1}t$ on the interval $\left(
-L_{1}t,\ell _{2}t\right) $ and with respect to the point $x=\ell _{2}t$ on
the interval $\left( -\ell _{1}t,L_{2}t\right) .$ In particular $\tilde{\phi}(-\ell _{1}t,t)=\tilde{\phi}(\ell
_{2}t,t)=0,$ hence the boundary condition remain satisfied, for every $t\geq
t_{0}$.

Derivation in time does not effect the odd-like symmetry in $x,$ hence $\phi
_{t}$ is extended as follows%
\begin{equation}
\tilde{\phi}_{t}(x,t)=\left\{ 
\begin{array}{ll}
-\frac{1-\ell _{1}}{1+\ell _{1}}\phi _{t}\left( -t+\frac{1-\ell _{1}}{1+\ell
_{1}}\left( t-x\right) ,t\right) , & x\in \left( -L_{1}t,-\ell _{1}t\right)
,\smallskip  \\ 
\phi _{t}\left( x,t\right) , & x\in \left( -\ell _{1}t,\ell _{2}t\right)
,\smallskip  \\ 
-\frac{1-\ell _{2}}{1+\ell _{2}}\phi _{t}\left( t-\frac{1-\ell _{2}}{1+\ell
_{2}}\left( t+x\right) ,t\right) ,\text{ \ \ \ \ } & x\in \left( \ell
_{2}t,L_{2}t\right) .%
\end{array}%
\right.   \label{phi1+}
\end{equation}%
Taking the derivative of (\ref{phi0+}) in $x$, we obtain%
\begin{equation}
\tilde{\phi}_{x}(x,t)=\left\{ 
\begin{array}{ll}
\frac{1-\ell _{1}}{1+\ell _{1}}\phi _{x}\left( -t+\frac{1-\ell _{1}}{1+\ell
_{1}}\left( t-x\right) ,t\right) , & x\in \left( -L_{1}t,-\ell _{1}t\right)
,\smallskip  \\ 
\phi _{x}\left( x,t\right) , & x\in \left( -\ell _{1}t,\ell _{2}t\right)
,\smallskip  \\ 
\frac{1-\ell _{2}}{1+\ell _{2}}\phi _{x}\left( t-\frac{1-\ell _{2}}{1+\ell
_{2}}\left( t+x\right) ,t\right) ,\text{ \ \ \ \ } & x\in \left( \ell
_{2}t,L_{2}t\right) .%
\end{array}%
\right.   \label{phi0x+}
\end{equation}%
Thus $\tilde{\phi}_{x}$ satisfies an even-like symmetry in $x$ with respect
to the point $x=-\ell _{1}t$ on the interval $\left( -L_{1}t,\ell
_{2}t\right) $ and with respect to the point $x=\ell _{2}t$ on the interval $%
\left( -\ell _{1}t,L_{2}t\right) .$

\begin{remark}
\label{rmkodd}If $\ell _{1}=0,$ then $L_{1}=\ell _{2}$ and $\tilde{\phi},%
\tilde{\phi}_{t}$ are odd functions and $\tilde{\phi}_{x}$ is an even
function on the interval $\left( -\ell _{2}t,\ell _{2}t\right) ,\forall
t\geq t_{0}.$ Similarly, if $\ell _{2}=0,$ then $L_{2}=\ell _{1}$ and $%
\tilde{\phi}^{0},\tilde{\phi}^{1}$ are odd function and $\tilde{\phi}_{x}$
is an even function on the interval $\left( -\ell _{1}t,\ell _{1}t\right) .$
This justifies the terminology \textquotedblleft odd-like\textquotedblright\
and \textquotedblleft even-like\textquotedblright\ used above.
\end{remark}

The next theorem shows how the calculate the coefficient of the series (\ref%
{exact0}).

\begin{theorem}
\label{thexist1}Under the assumptions \emph{(\ref{tlike}) }and \emph{(\ref%
{ic}),} The solution of Problem \emph{(\ref{wave})} is\emph{\ }given by the
series \emph{(\ref{exact0}). }The coefficients $c_{n}$ are given by any of
the two following formula%
\begin{eqnarray}
c_{n} &=&\frac{1}{2\pi in\sqrt{2\kappa _{\ell }}}\int\limits_{-\ell
_{1}t_{0}}^{L_{2}t_{0}}\left( \tilde{\phi}_{x}^{0}+\tilde{\phi}^{1}\right)
e^{-in\pi \kappa _{\ell }\log \left( t_{0}+x\right) }dx\text{, \ \ if\ }n\in 
\mathbb{Z}
^{\ast },  \label{cn+} \\
&=&\frac{e^{in\pi \kappa _{\ell }\log \frac{1+\ell _{1}}{1-\ell _{1}}}}{2\pi
in\sqrt{2\kappa _{\ell }}}\int\limits_{-L_{1}t_{0}}^{\ell _{2}t_{0}}\left( 
\tilde{\phi}_{x}^{0}-\tilde{\phi}^{1}\right) e^{-in\pi \kappa _{\ell }\log
\left( t_{0}-x\right) }dx\text{, \ \ if\ }n\in 
\mathbb{Z}
^{\ast }.  \label{cn-}
\end{eqnarray}%
where\emph{\ }$\tilde{\phi}^{1}$ and $\tilde{\phi}_{x}^{0}$ are given by 
\emph{(\ref{phi1+}) }and \emph{(\ref{phi0x+}). }Moreover, the sum $%
\sum_{n\in 
\mathbb{Z}
^{\ast }}\left\vert nc_{n}\right\vert ^{2}$ is finite and is given by any of
the two formula%
\begin{eqnarray}
\sum_{n\in 
\mathbb{Z}
^{\ast }}\left\vert nc_{n}\right\vert ^{2} &=&\frac{1}{8\pi ^{2}\kappa
_{\ell }}\int\limits_{-\ell _{1}t_{0}}^{L_{2}t_{0}}\left\vert \tilde{\phi}%
_{x}^{0}+\tilde{\phi}^{1}\right\vert ^{2}\frac{dx}{t_{0}+x}  \label{ncn+} \\
&=&\frac{1}{8\pi ^{2}\kappa _{\ell }}\int\limits_{-L_{1}t_{0}}^{\ell
_{2}t_{0}}\left\vert \tilde{\phi}_{x}^{0}-\tilde{\phi}^{1}\right\vert ^{2}%
\frac{dx}{t_{0}-x}.  \label{ncn-}
\end{eqnarray}
\end{theorem}

\begin{proof}
First, Let us denote%
\begin{equation}
C_{n}=c_{n}e^{in\pi \kappa _{\ell }\log \left( \frac{1+\ell _{2}}{1-\ell _{2}%
}\right) },\text{ \ \ }n\in 
\mathbb{Z}
^{\ast }  \label{cC}
\end{equation}%
or alternatively $C_{n}e^{in\pi \kappa _{\ell }\log \left( \frac{1+\ell _{1}%
}{1-\ell _{1}}\right) }=c_{n}$ since $e^{in\pi \kappa _{\ell }\log \alpha
_{\ell }\beta _{\ell }}=1.$ Due to (\ref{solreg}) and Remark \ref{rmk2}, we
can derive term by term the series (\ref{exact0}), it comes that, for $t\geq
t_{0},$%
\begin{equation}
\phi _{x}(x,t) =i\pi \kappa _{\ell }\sum_{n\in 
\mathbb{Z}
^{\ast }}n\left( c_{n}\frac{e^{in\pi \kappa _{\ell }\log \left( t+x\right) }%
}{t+x}+C_{n}\frac{e^{in\pi \kappa _{\ell }\log \left( t-x\right) }}{t-x}%
\right) ,\text{ \ \ }x\in \left( -\ell _{1}t,\ell _{2}t\right) \smallskip   \label{ph x} 
\end{equation}
\begin{equation}
\phi _{t}(x,t) =i\pi \kappa _{\ell }\sum_{n\in 
\mathbb{Z}
^{\ast }}n\left( c_{n}\frac{e^{in\pi \kappa _{\ell }\log \left( t+x\right) }%
}{t+x}-C_{n}\frac{e^{in\pi \kappa _{\ell }\log \left( t-x\right) }}{t-x}%
\right) ,\text{ \ \ }x\in \left( -\ell _{1}t,\ell _{2}t\right) .
\label{ph t}
\end{equation}%
Combining this, with (\ref{phi1+}) and (\ref{phi0x+}), for $t=t_{0},$ the
extensions of the initial data are given by%
\begin{equation}
\tilde{\phi}_{x}^{0}\left( x\right) =\left\{ 
\begin{array}{ll}
i\pi \kappa _{\ell }%
\displaystyle%
\sum_{n\in 
\mathbb{Z}
^{\ast }}n\left( c_{n}\dfrac{e^{in\pi \kappa _{\ell }\log \left( \frac{%
1-\ell _{1}}{1+\ell _{1}}\left( t_{0}-x\right) \right) }}{t_{0}-x}\right.  & 
\\ 
\multicolumn{1}{r}{\left. 
\displaystyle%
+\frac{1-\ell _{1}}{1+\ell _{1}}C_{n}\dfrac{e^{in\pi \kappa _{\ell }\log
\left( 2t_{0}-\frac{1-\ell _{1}}{1+\ell _{1}}\left( t_{0}-x\right) \right) }%
}{2t_{0}-\frac{1-\ell _{1}}{1+\ell _{1}}\left( t_{0}-x\right) }\right) ,} & 
x\in \left( -L_{1}t_{0},-\ell _{1}t_{0}\right) ,\smallskip  \\ 
i\pi \kappa _{\ell }%
\displaystyle%
\sum_{n\in 
\mathbb{Z}
^{\ast }}n\left( c_{n}\dfrac{e^{in\pi \kappa _{\ell }\log \left(
t_{0}+x\right) }}{t_{0}+x}+C_{n}\dfrac{e^{in\pi \kappa _{\ell }\log \left(
t_{0}-x\right) }}{t_{0}-x}\right) ,\text{\ } & x\in \left( -\ell
_{1}t_{0},\ell _{2}t_{0}\right) ,\smallskip  \\ 
i\pi \kappa _{\ell }%
\displaystyle%
\sum_{n\in 
\mathbb{Z}
^{\ast }}n\left( \frac{1-\ell _{2}}{1+\ell _{2}}c_{n}\dfrac{e^{in\pi \kappa
_{\ell }\log \left( 2t_{0}-\frac{1-\ell _{2}}{1+\ell _{2}}\left(
t_{0}+x\right) \right) }}{2t_{0}-\frac{1-\ell _{2}}{1+\ell _{2}}\left(
t_{0}+x\right) }\right.  &  \\ 
\multicolumn{1}{r}{%
\displaystyle%
\left. +C_{n}\dfrac{e^{in\pi \kappa _{\ell }\log \left( \frac{1-\ell _{2}}{%
1+\ell _{2}}\left( t_{0}+x\right) \right) }}{t_{0}+x}\right) ,} & x\in
\left( \ell _{2}t_{0},L_{2}t_{0}\right) .%
\end{array}%
\right.   \label{phx}
\end{equation}%
\begin{equation}
\tilde{\phi}^{1}(x)=\left\{ 
\begin{array}{ll}
i\pi \kappa _{\ell }%
\displaystyle%
\sum_{n\in 
\mathbb{Z}
^{\ast }}n\left( -c_{n}\dfrac{e^{in\pi \kappa _{\ell }\log \left( \frac{%
1-\ell _{1}}{1+\ell _{1}}\left( t_{0}-x\right) \right) }}{t_{0}-x}\right.  & 
\\ 
\multicolumn{1}{r}{\left. 
\displaystyle%
+\frac{1-\ell _{1}}{1+\ell _{1}}C_{n}\dfrac{e^{in\pi \kappa _{\ell }\log
\left( 2t_{0}-\frac{1-\ell _{1}}{1+\ell _{1}}\left( t_{0}-x\right) \right) }%
}{2t_{0}-\frac{1-\ell _{1}}{1+\ell _{1}}\left( t_{0}-x\right) }\right) ,} & 
x\in \left( -L_{1}t_{0},-\ell _{1}t_{0}\right) ,\smallskip  \\ 
i\pi \kappa _{\ell }%
\displaystyle%
\sum_{n\in 
\mathbb{Z}
^{\ast }}n\left( c_{n}\dfrac{e^{in\pi \kappa _{\ell }\log \left(
t_{0}+x\right) }}{t_{0}+x}-C_{n}\dfrac{e^{in\pi \kappa _{\ell }\log \left(
t_{0}-x\right) }}{t_{0}-x}\right) ,\text{ } & x\in \left( -\ell
_{1}t_{0},\ell _{2}t_{0}\right) ,\smallskip  \\ 
i\pi \kappa _{\ell }%
\displaystyle%
\displaystyle%
\sum_{n\in 
\mathbb{Z}
^{\ast }}n\left( -\frac{1-\ell _{2}}{1+\ell _{2}}c_{n}\dfrac{e^{in\pi \kappa
_{\ell }\log \left( 2t_{0}-\frac{1-\ell _{2}}{1+\ell _{2}}\left(
t_{0}+x\right) \right) }}{2t_{0}-\frac{1-\ell _{2}}{1+\ell _{2}}\left(
t_{0}+x\right) }\right.  &  \\ 
\multicolumn{1}{r}{\left. +C_{n}\dfrac{e^{in\pi \kappa _{\ell }\log \left( 
\frac{1-\ell _{2}}{1+\ell _{2}}\left( t_{0}+x\right) \right) }}{t_{0}+x}%
\right) ,} & x\in \left( \ell _{2}t_{0},L_{2}t_{0}\right) .%
\end{array}%
\right.   \label{pht}
\end{equation}

On one hand, taking the sum of (\ref{phx}) and (\ref{pht}), we get%
\begin{equation*}
\tilde{\phi}_{x}^{0}+\tilde{\phi}^{1}=\left\{ 
\begin{array}{ll}
2\pi i\kappa _{\ell }\frac{1-\ell _{1}}{1+\ell _{1}}%
\displaystyle%
\sum_{n\in 
\mathbb{Z}
^{\ast }}nC_{n}\dfrac{e^{in\pi \kappa _{\ell }\log \left( 2t_{0}-\frac{%
1-\ell _{1}}{1+\ell _{1}}\left( t_{0}-x\right) \right) }}{2t_{0}-\frac{%
1-\ell _{1}}{1+\ell _{1}}\left( t_{0}-x\right) }, & x\in \left(
-L_{1}t_{0},-\ell _{1}t_{0}\right) ,\smallskip \\ 
2\pi i\kappa _{\ell }%
\displaystyle%
\sum_{n\in 
\mathbb{Z}
^{\ast }}nc_{n}\dfrac{e^{in\pi \kappa _{\ell }\log \left( t_{0}+x\right) }}{%
t_{0}+x}, & x\in \left( -\ell _{1}t_{0},\ell _{2}t_{0}\right) ,\smallskip \\ 
2\pi i\kappa _{\ell }\frac{1+\ell _{2}}{1-\ell _{2}}%
\displaystyle%
\sum_{n\in 
\mathbb{Z}
^{\ast }}nC_{n}\dfrac{e^{in\pi \kappa _{\ell }\log \left( \frac{1-\ell _{2}}{%
1+\ell _{2}}\left( t_{0}+x\right) \right) }}{t_{0}+x}, & x\in \left( \ell
_{2}t_{0},L_{2}t_{0}\right) .%
\end{array}%
\right.
\end{equation*}%
In particular,%
\begin{equation*}
\left( t_{0}+x\right) \left( \tilde{\phi}_{x}^{0}+\tilde{\phi}^{1}\right)
=\left\{ 
\begin{array}{ll}
2\pi i\kappa _{\ell }%
\displaystyle%
\sum_{n\in 
\mathbb{Z}
^{\ast }}nc_{n}e^{in\pi \kappa _{\ell }\log \left( t_{0}+x\right) }, & x\in
\left( -\ell _{1}t_{0},\ell _{2}t_{0}\right) ,\smallskip \\ 
2\pi i\kappa _{\ell }%
\displaystyle%
\sum_{n\in 
\mathbb{Z}
^{\ast }}\left( nC_{n}e^{in\pi \kappa _{\ell }\log \left( \frac{1+\ell _{1}}{%
1-\ell _{1}}\right) }\right) e^{in\pi \kappa _{\ell }\log \left(
t_{0}+x\right) }, & x\in \left( \ell _{2}t_{0},L_{2}t_{0}\right) .%
\end{array}%
\right.
\end{equation*}%
Due to (\ref{cC}), we can write 
\begin{equation}
\frac{1}{2\pi i\sqrt{2\kappa _{\ell }}}\left( t_{0}+x\right) \left( \tilde{%
\phi}_{x}^{0}+\tilde{\phi}^{1}\right) =\sum_{n\in 
\mathbb{Z}
^{\ast }}nc_{n}\sqrt{\kappa _{\ell }/2}e^{in\pi \kappa _{\ell }\log \left(
t_{0}+x\right) },\text{ \ \ for }x\in \left( -\ell
_{1}t_{0},L_{2}t_{0}\right) .  \label{30}
\end{equation}%
Thanks to Lemma \ref{lm2}, we see that $nc_{n}$ is the $n^{th}$ coefficient
of the function%
\begin{equation}
\frac{1}{2\pi i\sqrt{2\kappa _{\ell }}}\left( t_{0}+x\right) \left( \tilde{%
\phi}_{x}^{0}+\tilde{\phi}^{1}\right) \in L^{2}\left( -\ell
_{1}t_{0},L_{2}t_{0},\frac{dx}{t_{0}+x}\right)  \label{parsev+}
\end{equation}%
in the basis $\left\{ \sqrt{\kappa _{\ell }/2}e^{in\pi \kappa _{\ell }\log
\left( t_{0}+x\right) }\right\} _{n\in 
\mathbb{Z}
}$, i.e.%
\begin{equation*}
nc_{n}=\frac{1}{2\pi i\sqrt{2\kappa _{\ell }}}\int\limits_{-\ell
_{1}t_{0}}^{L_{2}t_{0}}\left( t_{0}+x\right) \left( \tilde{\phi}_{x}^{0}+%
\tilde{\phi}^{1}\right) e^{-in\pi \kappa _{\ell }\log \left( t_{0}+x\right) }%
\frac{dx}{t_{0}+x}\text{, \ \ \ for\ }n\in 
\mathbb{Z}
^{\ast }
\end{equation*}%
and (\ref{cn+}) holds as claimed.

On the other hand, taking the difference of (\ref{phx}) and (\ref{pht}), we
get in particular%
\begin{equation*}
\tilde{\phi}_{x}^{0}-\tilde{\phi}^{1}=\left\{ 
\begin{array}{ll}
2i\pi \kappa _{\ell }%
\displaystyle%
\sum_{n\in 
\mathbb{Z}
^{\ast }}nc_{n}\frac{e^{in\pi \kappa _{\ell }\log \left( \frac{1-\ell _{1}}{%
1+\ell _{1}}\left( t_{0}-x\right) \right) }}{\left( t_{0}-x\right) }, & x\in
\left( -L_{1}t_{0},-\ell _{1}t_{0}\right) ,\smallskip  \\ 
2i\pi \kappa _{\ell }%
\displaystyle%
\sum_{n\in 
\mathbb{Z}
^{\ast }}nC_{n}\frac{e^{in\pi \kappa _{\ell }\log \left( t_{0}-x\right) }}{%
\left( t_{0}-x\right) }, & x\in \left( -\ell _{1}t_{0},\ell _{2}t_{0}\right)
,\smallskip 
\end{array}%
\right. 
\end{equation*}%
hence%
\begin{equation*}
\left( t_{0}-x\right) \left( \tilde{\phi}_{x}^{0}-\tilde{\phi}^{1}\right)
=\left\{ 
\begin{array}{ll}
2i\pi \kappa _{\ell }%
\displaystyle%
\sum_{n\in 
\mathbb{Z}
^{\ast }}n\left( c_{n}e^{in\pi \kappa _{\ell }\log \left( \frac{1-\ell _{1}}{%
1+\ell _{1}}\right) }\right) e^{in\pi \kappa _{\ell }\log \left(
t_{0}-x\right) }, & x\in \left( -L_{1}t_{0},-\ell _{1}t_{0}\right)
,\smallskip  \\ 
i\pi \kappa _{\ell }%
\displaystyle%
\sum_{n\in 
\mathbb{Z}
^{\ast }}nC_{n}e^{in\pi \kappa _{\ell }\log \left( \left( t_{0}-x\right)
\right) }, & x\in \left( -\ell _{1}t_{0},\ell _{2}t_{0}\right) ,\smallskip 
\end{array}%
\right. 
\end{equation*}%
Using (\ref{cC}), we can write%
\begin{equation}
\frac{1}{2\pi i\sqrt{2\kappa _{\ell }}}\left( t_{0}-x\right) \left( \tilde{%
\phi}_{x}^{0}-\tilde{\phi}^{1}\right) =\sum_{n\in 
\mathbb{Z}
^{\ast }}nC_{n}\sqrt{\kappa _{\ell }/2}e^{in\pi \kappa _{\ell }\log \left(
t_{0}-x\right) },\text{\ \ \ for }x\in \left( -L_{1}t_{0},\ell
_{2}t_{0}\right) .  \label{31}
\end{equation}%
This means that $nC_{n}$ is the $n^{th}$ coefficient of the function%
\begin{equation}
\frac{1}{2\pi i\sqrt{2\kappa _{\ell }}}\left( t_{0}-x\right) \left( \tilde{%
\phi}_{x}^{0}-\tilde{\phi}^{1}\right) \in L^{2}\left( -L_{1}t_{0},\ell
_{2}t_{0},\frac{dx}{t_{0}-x}\right)   \label{parsev-}
\end{equation}%
in the basis $\left\{ \sqrt{\kappa _{\ell }/2}e^{in\pi \kappa _{\ell }\log
\left( t_{0}-x\right) }\right\} _{n\in 
\mathbb{Z}
}$, i.e. (\ref{cn-}) also holds.

The sums $\sum_{n\in 
\mathbb{Z}
^{\ast }}\left\vert nc_{n}\right\vert ^{2}$ and $\sum_{n\in 
\mathbb{Z}
^{\ast }}\left\vert nC_{n}\right\vert ^{2}$ result from Parseval's equality
applied to the functions given in (\ref{parsev+}) and (\ref{parsev-})
respectively, hence (\ref{ncn+}) and (\ref{ncn-}) follows since $\left\vert
c_{n}\right\vert =\left\vert C_{n}\right\vert $.
\end{proof}

\section{Energy estimates}

The following theorem gives the precise decay rate of the energy for the
solution of (\ref{wave}).

\begin{theorem}
\label{th1}Under the assumptions \emph{(\ref{tlike})} and \emph{(\ref{ic}),}
the solution of Problem \emph{(\ref{wave})} satisfies%
\begin{equation}
tE\left( t\right) +\int\limits_{-\ell _{1}t}^{\ell _{2}t}x\phi _{x}\phi
_{t}\ dx=S_{\ell },\text{ \ \ for \ }t\geq t_{0},  \label{est0}
\end{equation}%
where $S_{\ell }:=2\pi ^{2}\kappa _{\ell }\sum_{n\in 
\mathbb{Z}
^{\ast }}\left\vert nc_{n}\right\vert ^{2},$ and it holds that%
\begin{equation}
\frac{S_{\ell }}{t\left( 1+L\right) }\leq E\left( t\right) \leq \frac{%
S_{\ell }}{t\left( 1-L\right) },\ \ \ \ \ \text{\ for }t\geq t_{0}.
\label{ES}
\end{equation}
\end{theorem}

\begin{proof}
On one hand, due to Lemma \ref{lm02}, (\ref{30}) and Parseval's equality
applied to 
\begin{equation*}
(t-x)\left( \tilde{\phi}_{x}-\tilde{\phi}_{t}\right) \in L^{2}\left(
-L_{1}t,\ell _{2}t,\frac{dx}{t-x}\right) ,
\end{equation*}%
we infer that%
\begin{equation}
\int\limits_{-L_{1}t}^{\ell _{2}t}(t-x)\left( \tilde{\phi}_{x}-\tilde{\phi}%
_{t}\right) ^{2}dx=\int\limits_{-L_{1}t}^{\ell _{2}t}\left\vert (t-x)\left( 
\tilde{\phi}_{x}-\tilde{\phi}_{t}\right) \right\vert ^{2}\frac{dx}{t-x}=8\pi
^{2}\kappa \sum_{n\in 
\mathbb{Z}
^{\ast }}\left\vert nC_{n}\right\vert ^{2}=4S_{\ell }.  \label{E3}
\end{equation}%
Similarly, due to (\ref{31}) and Parseval's equality, applied to $%
(t+x)\left( \tilde{\phi}_{x}+\tilde{\phi}_{t}\right) \in L^{2}\left( -\ell
_{1}t,L_{2}t,\frac{dx}{t+x}\right) $, we obtain%
\begin{equation}
\int\limits_{-\ell _{1}t}^{L_{2}t}(t+x)\left( \tilde{\phi}_{x}+\tilde{\phi}%
_{t}\right) ^{2}dx=\int\limits_{-\ell _{1}t}^{L_{2}t}\left\vert (t+x)\left( 
\tilde{\phi}_{x}+\tilde{\phi}\right) \right\vert ^{2}\frac{dx}{t+x}=8\pi
^{2}\kappa \sum_{n\in 
\mathbb{Z}
^{\ast }}\left\vert nc_{n}\right\vert ^{2}=4S_{\ell }.  \label{E2}
\end{equation}%
Using (\ref{phi0x+}), (\ref{phi1+}) and considering the change of variable $%
x=t-\frac{1-\ell _{1}}{1+\ell _{1}}\left( t+\xi \right) $, in $\left(
-L_{1}t,-\ell _{1}t\right) ,$ we obtain%
\begin{eqnarray*}
\int\limits_{-L_{1}t}^{-\ell _{1}t}(t-x)\left( \tilde{\phi}_{x}\left(
x\right) -\tilde{\phi}_{t}\left( x\right) \right) ^{2}dx &=&-\left( \frac{%
1-\ell _{1}}{1+\ell _{1}}\right) ^{2}\int\limits_{\ell _{2}t}^{-\ell
_{1}t}\left( t+\xi \right) \left( \left( \frac{1+\ell _{1}}{1-\ell _{1}}%
\right) \left( \phi _{x}\left( \xi \right) +\phi _{t}\left( \xi \right)
\right) \right) ^{2}d\xi \\
&=&\int\limits_{-\ell _{1}t}^{\ell _{2}t}\left( t+\xi \right) \left( \left(
\phi _{x}\left( \xi \right) +\phi _{t}\left( \xi \right) \right) \right)
^{2}d\xi .
\end{eqnarray*}%
Taking $x=t-\frac{1-\ell _{2}}{1+\ell _{2}}\left( t+\xi \right) $ in $\left(
\ell _{2}t,L_{2}t\right) ,$ we obtain%
\begin{equation*}
\int\limits_{\ell _{2}t}^{L_{2}t}(t+x)\left( \tilde{\phi}_{x}\left( x\right)
+\tilde{\phi}_{t}\left( x\right) \right) ^{2}dx=\int\limits_{-\ell
_{1}t}^{\ell _{2}t}\left( t-\xi \right) \left( \phi _{x}\left( \xi \right)
-\phi _{t}\left( \xi \right) \right) ^{2}d\xi .
\end{equation*}%
Thus, considering the sum of (\ref{E3}) and (\ref{E2}), we infer that%
\begin{multline*}
\int\limits_{-L_{1}t}^{\ell _{2}t}(t+x)\left( \tilde{\phi}_{t}+\tilde{\phi}%
_{x}\right) ^{2}dx+\int\limits_{-\ell _{1}t}^{L_{2}t}(t-x)\left( \tilde{\phi}%
_{x}-\tilde{\phi}_{t}\right) ^{2}dx \\
=2\int\limits_{-\ell _{1}t}^{\ell _{2}t}\left( t+x\right) \left( \phi
_{x}+\phi _{t}\right) ^{2}dx+2\int\limits_{-\ell _{1}t}^{\ell
_{2}t}(t-x)\left( \phi _{t}-\phi _{x}\right) ^{2}dx=8S_{\ell }.
\end{multline*}%
Expanding $\left( \phi _{x}\pm \phi _{t}\right) ^{2}$ and collecting similar
terms, we get 
\begin{equation*}
4t\int\limits_{-\ell _{1}t}^{\ell _{2}t}\phi _{x}^{2}+\phi _{t}^{2}\
dx+8\int\limits_{-\ell _{1}t}^{\ell _{2}t}x\phi _{x}\phi _{t}\ dx=8S_{\ell },%
\text{ \ \ for}\ t\geq t_{0}.
\end{equation*}%
Recalling that $E\left( t\right) $ is given by (\ref{E}), then (\ref{est0})
holds as claimed.

Finally, we use the algebraic inequality $\pm ab\leq \left(
a^{2}+b^{2}\right) /2$ and $\left\vert x\right\vert \leq L$ to obtain%
\begin{equation*}
\pm \int\limits_{-\ell _{1}t}^{\ell _{2}t}x\phi _{x}\phi _{t}\ dx\leq Lt%
\text{ }E\left( t\right) \text{, \ \ \ for }t\geq t_{0},
\end{equation*}%
Due to (\ref{est0}), it comes that%
\begin{equation}
S_{\ell }\leq \left( 1+L\right) t\text{ }E\left( t\right) \text{ \ \ and \ \ 
}\left( 1-L\right) t\text{ }E\left( t\right) \leq S_{\ell },\text{ \ for }%
t\geq t_{0},  \label{ESES}
\end{equation}%
This implies (\ref{ES}) and the theorem follows.
\end{proof}

The next corollary compares $E\left( t\right) $\ to the initial energy $%
E\left( t_{0}\right) .$

\begin{corollary}
Under the assumptions \emph{(\ref{tlike}) }and \emph{(\ref{ic}),} the energy
of the solution of Problem \emph{(\ref{wave})} satisfies%
\begin{equation}
\left( \frac{1-L}{1+L}\right) \frac{t_{0}E\left( t_{0}\right) }{t}\leq
E\left( t\right) \leq \left( \frac{1+L}{1-L}\right) \frac{t_{0}E\left(
t_{0}\right) }{t},\ \ \ \ \ \text{\ for }t\geq t_{0}.  \label{stab}
\end{equation}
\end{corollary}

\begin{proof}
Since (\ref{ESES}) holds also for $t=t_{0}$, then \emph{(\ref{stab})}
follows by combining the two inequalities%
\begin{equation*}
\left( 1-L\right) t\text{ }E\left( t\right) \leq \left( 1+L\right) t_{0}%
\text{ }E\left( t_{0}\right) \ \ \text{and}\ \ \left( 1+L\right) t\text{ }%
E\left( t\right) \geq \left( 1-L\right) t_{0}\text{ }E\left( t_{0}\right) ,
\end{equation*}%
for\ $t\geq t_{0}.$
\end{proof}

\begin{remark}
Since $E\left( t\right) $ define a norm on the space $H_{0}^{1}\left(
I_{t}\right) \times L^{2}\left( I_{t}\right) ,$ for $t\geq 0,$ then (\ref%
{stab})\emph{\ }implies the uniqueness for the solution of\emph{\ (\ref{wave}%
).}
\end{remark}

\section{Observability and controllability at one endpoint}

In this section, we show the observability of (\ref{wave}) at one of the
endpoints, say $x=\ell _{2}t$, then by applying HUM we deduce the exact
boundary controllability for (\ref{wavec}).

\subsection{Observability at one endpoint}

First, we can state the following lemma.

\begin{lemma}
\label{lm2}Let $M$ be a positive integer and $\phi $ the solution of \emph{(%
\ref{wave}). }Under the assumptions \emph{(\ref{tlike}) }and \emph{(\ref{ic}%
),} we have 
\begin{equation}
\int_{t_{0}}^{\left( \alpha _{\ell }\beta _{\ell }\right) ^{M}t_{0}}t\phi
_{x}^{2}(\ell _{2}t,t)dt=\frac{4M}{\left( 1-\ell _{2}^{2}\right) ^{2}}%
S_{\ell },  \label{=slm}
\end{equation}%
and it holds that%
\begin{equation}
\frac{4M\left( 1-L\right) t_{0}}{\left( 1-\ell _{2}^{2}\right) ^{2}}E\left(
t_{0}\right) \leq \int_{t_{0}}^{\left( \alpha _{\ell }\beta _{\ell }\right)
^{M}t_{0}}t\phi _{x}^{2}(\ell _{2}t,t)dt\leq \frac{4M\left( 1+L\right) t_{0}%
}{\left( 1-\ell _{2}^{2}\right) ^{2}}E\left( t_{0}\right) .  \label{EBlt}
\end{equation}
\end{lemma}

\begin{proof}
Thanks to (\ref{ph x}), we can evaluate $\phi _{x}$ at the endpoint $x=\ell
_{2}t$. We obtain%
\begin{eqnarray}
\phi _{x}(\ell _{2}t,t) &=&\pi \kappa _{\ell }\sum_{n\in 
\mathbb{Z}
^{\ast }}inc_{n}\left( \dfrac{e^{in\pi \kappa _{\ell }\log \left( 1+\ell
_{2}\right) t}}{\left( 1+\ell _{2}\right) t}+\dfrac{e^{in\pi \kappa _{\ell
}\log \left( 1+\ell _{2}\right) t}}{\left( 1-\ell _{2}\right) t}\right) 
\notag \\
&=&\pi \kappa _{\ell }\sum_{n\in 
\mathbb{Z}
^{\ast }}inc_{n}\left( \dfrac{1}{\left( 1+\ell _{2}\right) t}+\dfrac{1}{%
\left( 1-\ell _{2}\right) t}\right) e^{in\pi \kappa _{\ell }\log \left(
1+\ell _{2}\right) t},  \notag
\end{eqnarray}%
hence%
\begin{equation}
t\phi _{x}(\ell _{2}t,t)=\frac{2\pi \kappa _{\ell }}{\left( 1-\ell
_{2}^{2}\right) }\sum_{n\in 
\mathbb{Z}
^{\ast }}inc_{n}e^{in\pi \kappa _{\ell }\log \left( 1+\ell _{2}\right) }\
e^{in\pi \kappa _{\ell }\log t}.  \label{tphx+}
\end{equation}%
By Lemma \ref{lm01} and Parseval's equality applied to the function $t\phi
_{x}(\ell _{2}t,t)\in L^{2}\left( t_{0},\left( \alpha _{\ell }\beta _{\ell
}\right) ^{M}t_{0},\frac{dt}{t}\right) $, we obtain 
\begin{equation*}
\int_{t_{0}}^{\left( \alpha _{\ell }\beta _{\ell }\right) ^{M}t_{0}}t\phi
_{x}^{2}(\ell _{2}t,t)dt=\frac{4\pi ^{2}\kappa _{\ell }^{2}}{\left( 1-\ell
_{2}^{2}\right) ^{2}}\left( \frac{2M}{\kappa _{\ell }}\right) \sum_{n\in 
\mathbb{Z}
^{\ast }}\left\vert nc_{n}e^{in\pi \kappa _{\ell }\log \left( 1+\ell
_{2}\right) }\right\vert ^{2}=\frac{8M\pi ^{2}\kappa _{\ell }}{\left( 1-\ell
_{2}^{2}\right) ^{2}}\sum_{n\in 
\mathbb{Z}
^{\ast }}\left\vert nc_{n}\right\vert ^{2}
\end{equation*}%
which is (\ref{=slm}). The estimate (\ref{EBlt}) follows by using (\ref{ESES}%
) for $t=t_{0}$.
\end{proof}

The next theorem gives the direct and inverse inequalities for the solution
of Problem (\ref{wave}) at the endpoint $x=\ell _{2}t$.

\begin{theorem}
\label{thobs1}Under the assumptions \emph{(\ref{tlike}) }and \emph{(\ref{ic}%
),} we have:

$\bullet $ For every $T\geq 0,$ the solution of \emph{(\ref{wave})}
satisfies the direct inequality 
\begin{equation}
\int_{t_{0}}^{t_{0}+T}\phi _{x}^{2}(\ell _{2}t,t)dt\leq K_{\ell }\left(
T\right) \text{ }E\left( t_{0}\right) .  \label{D1}
\end{equation}%
with a constant $K_{\ell }\left( T\right) $ depending only on $\ell
_{1},\ell _{2}$ and $T.$

$\bullet $ If $T\geq T_{\ell },$ Problem \emph{(\ref{wave})} is observable
at $x=\ell _{2}t$ and it holds that 
\begin{equation}
E\left( t_{0}\right) \leq \frac{\alpha _{\ell }\beta _{\ell }\left( 1-\ell
_{2}^{2}\right) ^{2}}{4\left( 1-L\right) }\int_{t_{0}}^{t_{0}+T}\phi
_{x}^{2}(\ell _{2}t,t)dt.  \label{obs1}
\end{equation}%
Conversely, if $T<T_{\ell },$ \emph{(\ref{obs1})} does not hold.
\end{theorem}

\begin{proof}
$\bullet $ The right-hand side of (\ref{EBlt}) yields%
\begin{equation*}
t_{0}\int_{t_{0}}^{\left( \alpha _{\ell }\beta _{\ell }\right)
^{M}t_{0}}\phi _{x}^{2}(\ell _{2}t,t)dt\leq \frac{4M\left( 1+L\right) t_{0}}{%
\left( 1-\ell _{2}^{2}\right) ^{2}}E\left( t_{0}\right) .
\end{equation*}%
Noting that $\alpha _{\ell }\beta _{\ell }>1,$ then for every $T\geq 0,$ we
can take the integer $M$ large enough such that%
\begin{equation*}
\left( \alpha _{\ell }\beta _{\ell }\right) ^{M}t_{0}\geq t_{0}+T.
\end{equation*}%
For instance we can take $M\geq \log \left( 1+\left( T/t_{0}\right) \right)
/\log \left( \alpha _{\ell }\beta _{\ell }\right) .$ Since the integrated
function is nonnegative, then (\ref{EBlt}) holds for $K_{\ell }\left(
T\right) :=4M\left( 1+L\right) /\left( 1-\ell _{2}^{2}\right) ^{2}.$

$\bullet $ The left-hand side of (\ref{EBlt}), for $M=1,$ yields%
\begin{equation*}
\frac{4\left( 1-L\right) t_{0}}{\left( 1-\ell _{2}^{2}\right) ^{2}}E\left(
t_{0}\right) \leq \alpha _{\ell }\beta _{\ell }t_{0}\int_{t_{0}}^{\alpha
_{\ell }\beta _{\ell }t_{0}}\phi _{x}^{2}(\ell _{2}t,t)dt
\end{equation*}%
and thus inequality (\ref{obs1}) holds for $T=\alpha _{\ell }\beta _{\ell
}t_{0}-t_{0}=T_{\ell }$ and therefore for every $T\geq T_{\ell }$ as well.

It remain to show that the observability does not hold for every $T<T_{\ell
}.$ Let $\varepsilon >0,\varepsilon \in \left( 0,1\right) ,$ and consider a
function $g\in L^{2}\left( t_{0},\alpha _{\ell }\beta _{\ell
}t_{0},dt/t\right) ,$ not identically zero and satisfying 
\begin{gather}
supp\left( g\right) \subset \subset \left( \left( 1-\varepsilon \right)
\alpha _{\ell }\beta _{\ell }t_{0},\alpha _{\ell }\beta _{\ell }t_{0}\right)
,\text{ \ \ for }t\in \left( t_{0},\alpha _{\ell }\beta _{\ell }t_{0}\right) 
\label{suppg} \\
\int_{t_{0}}^{\alpha _{\ell }\beta _{\ell }t_{0}}g\left( t\right) \frac{dt}{t%
}=0.  \label{intgg}
\end{gather}%
Since $\left\{ \sqrt{\kappa _{\ell }/2}\ e^{in\pi \kappa _{\ell }\log
t}\right\} _{n\in 
\mathbb{Z}
}$ is a complete basis for $L^{2}\left( t_{0},\alpha _{\ell }\beta _{\ell
}t_{0},dt/t\right) $, then $g$ can be written as%
\begin{equation}
g\left( t\right) =\sum_{n\in 
\mathbb{Z}
}g_{n}e^{in\pi \kappa _{\ell }\log t},\text{ \ \ }t\in \left( t_{0},\alpha
_{\ell }\beta _{\ell }t_{0}\right) .  \label{gseries}
\end{equation}%
Note that (\ref{intgg}) ensures that $g_{0}=0.$ Let us define 
\begin{equation}
\mathbf{c}_{0}=0,\text{ \ and \ }\mathbf{c}_{n}=\frac{g_{n}}{i2\pi n\sqrt{%
2\kappa _{\ell }}}e^{-in\pi \kappa _{\ell }\log \left( 1+\ell _{2}\right) },%
\text{ \ \ for }n\in 
\mathbb{Z}
^{\ast }.  \label{cng}
\end{equation}%
The sequence $\left\{ \mathbf{c}_{n}\right\} _{n\in 
\mathbb{Z}
}$ satisfies $\left\vert n\mathbf{c}_{n}\right\vert \leq \left\vert
g_{n}/\left( i2\pi \sqrt{2\kappa _{\ell }}\right) \right\vert ,$ then the
series%
\begin{equation*}
\sum_{n\in 
\mathbb{Z}
^{\ast }}i2\pi n\sqrt{2\kappa _{\ell }}\mathbf{c}_{n}\text{ }e^{i\pi n\kappa
_{\ell }\log \left( t_{0}+x\right) }
\end{equation*}%
is converging and define a non-identically zero function $\mathbf{g}\in
L^{2}\left( -\ell _{1}t_{0},L_{2}t_{0},dx/\left( t_{0}+x\right) \right) .$
Having (\ref{30}) in mind, we need to choose two functions $\tilde{\phi}%
_{x}^{0}\in L^{2}\left( -\ell _{1}t_{0},L_{2}t_{0}\right) $ and $\tilde{\phi}%
^{1}\in L^{2}\left( -\ell _{1}t_{0},L_{2}t_{0}\right) $, satisfying
respectively the even-like and odd-like symmetries considered in (\ref%
{phi0x+}) and (\ref{phi1+}), with respect to $x=\ell _{2}t_{0},\ $and such
that 
\begin{equation}
\left( t_{0}+x\right) \left( \tilde{\phi}_{x}^{0}+\tilde{\phi}^{1}\right) =%
\mathbf{g(x)},\text{ \ \ for }\ x\in \left( -\ell
_{1}t_{0},L_{2}t_{0}\right) ,  \label{icchoice}
\end{equation}%
i.e.,%
\begin{equation}
\tilde{\phi}_{x}^{0}(x)+\tilde{\phi}^{1}(x)=\mathbf{G}\left( x\right) :=%
\frac{\mathbf{g(x)}}{t_{0}+x},\text{ \ \ for }\ x\in \left( -\ell
_{1}t_{0},L_{2}t_{0}\right) .  \label{gw}
\end{equation}%
This is realised in the following way. Consider the function $\mathbf{\tilde{%
G}}$ defined as$\ $%
\begin{equation*}
\mathbf{\tilde{G}}(x)=\left\{ 
\begin{array}{ll}
\frac{1+\ell _{2}}{1-\ell _{2}}\mathbf{G}\left( -t_{0}+\frac{1+\ell _{2}}{%
1-\ell _{2}}\left( t_{0}-x\right) \right) , & x\in \left( -\ell
_{1}t_{0},\ell _{2}t_{0}\right) ,\smallskip  \\ 
\frac{1-\ell _{2}}{1+\ell _{2}}\mathbf{G}\left( t_{0}-\frac{1-\ell _{2}}{%
1+\ell _{2}}\left( t_{0}+x\right) \right) ,\text{ \ } & x\in \left( \ell
_{2}t_{0},L_{2}t_{0}\right) ,%
\end{array}%
\right. 
\end{equation*}%
then it suffices to take%
\begin{equation*}
\tilde{\phi}_{x}^{0}\left( x\right) =\frac{\mathbf{G}(x)+\mathbf{\tilde{G}}%
(x)}{2}\text{ \ \ and \ \ }\tilde{\phi}^{1}\left( x\right) =\frac{\mathbf{G}%
(x)-\mathbf{\tilde{G}}(x)}{2},\text{ \ \ for }\ x\in \left( -\ell
_{1}t_{0},L_{2}t_{0}\right) .
\end{equation*}%
One can check that $\tilde{\phi}_{x}^{0}$ and $\tilde{\phi}^{1}$ are
respectively an even-like and an odd-like function on $\left( -\ell
_{1}t_{0},L_{2}t_{0}\right) >$ In particular  
\begin{equation*}
\int\limits_{-\ell _{1}t_{0}}^{L_{2}t_{0}}\tilde{\phi}^{1}\left( x\right)
dx=0
\end{equation*}%
and by (\ref{gw}) it follows that%
\begin{equation*}
\tilde{\phi}^{0}\left( \ell _{2}t_{0}\right) -\tilde{\phi}^{0}\left( -\ell
_{1}t_{0}\right) =\int\limits_{-\ell _{1}t_{0}}^{\ell _{2}t_{0}}\tilde{\phi}%
_{x}^{0}\left( x\right) dx=\frac{1}{2}\int\limits_{-\ell
_{1}t_{0}}^{L_{2}t_{0}}\tilde{\phi}_{x}^{0}\left( x\right) dx=\frac{1}{2}%
\int\limits_{-\ell _{1}t_{0}}^{L_{2}t_{0}}\mathbf{g}(x)\frac{dx}{t_{0}+x},
\end{equation*}%
i.e. $\tilde{\phi}^{0}\left( \ell _{2}t_{0}\right) -\tilde{\phi}^{0}\left(
-\ell _{1}t_{0}\right) =0\times \mathbf{c}_{0}=0.$ Thus, we can always
assume that $\tilde{\phi}^{0}\ $is in $H_{0}^{1}\left( -\ell
_{1}t_{0},L_{2}t_{0}\right) .$

Going back to (\ref{tphx+}), for $M=1,$ the solution of Problem (\ref{wave})
corresponding to these initial conditions satisfies%
\begin{equation*}
\phi _{x}(\ell _{2}t,t)=\frac{1}{\left( 1-\ell _{2}^{2}\right) t}\sum_{n\in 
\mathbb{Z}
^{\ast }}i2\pi \kappa _{\ell }n\mathbf{c}_{n}e^{in\pi \kappa _{\ell }\log
\left( 1+\ell _{2}\right) }\ e^{in\pi \kappa _{\ell }\log t}=\frac{1}{\left(
1-\ell _{2}^{2}\right) t}\sum_{n\in 
\mathbb{Z}
^{\ast }}g_{n}e^{in\pi \kappa _{\ell }\log t},
\end{equation*}%
hence 
\begin{equation*}
\phi _{x}(\ell _{2}t,t)=\frac{g\left( t\right) }{\left( 1-\ell
_{2}^{2}\right) t}.
\end{equation*}
Squaring both sides and integrating on $\left( t_{0},\left( 1-\varepsilon
\right) \alpha _{\ell }\beta _{\ell }t_{0}\right) ,$ we get 
\begin{equation*}
\int_{t_{0}}^{\left( 1-\varepsilon \right) \alpha _{\ell }\beta _{\ell
}t_{0}}\phi _{x}^{2}(\ell _{2}t,t)dt=\frac{1}{\left( 1-\ell _{2}^{2}\right)
^{2}}\int_{t_{0}}^{\left( 1-\varepsilon \right) \alpha _{\ell }\beta _{\ell
}t_{0}}g^{2}\left( t\right) \frac{dt}{t^{2}}=0,
\end{equation*}%
since supp$\left( g\right) \cap \left( t_{0},\left( 1-\varepsilon \right)
\alpha _{\ell }\beta _{\ell }t_{0}\right) =\varnothing .$ This means that
the observability inequality (\ref{obs1}) does not hold for every $T<T_{\ell
}$.
\end{proof}

\begin{remark}
We obtain the same time of observability at the left boundary $x=-\ell _{1}t$%
. The proof is parallel to the precedent one.
\end{remark}

\begin{remark}
\label{rmktobs}The time of observability $T_{\ell }$ can be predicted by a
simple argument, see Figure $\ref{fig3}$. An initial disturbance
concentrated near $x=\ell _{2}t_{0}$ may propagate to the left, as $t$
increases, and bounce back on the left boundary, then travel to reach the
left boundary, when $t$ is close to $\alpha _{\ell }\beta _{\ell }t_{0}$,
see \emph{Figure }$\ref{fig3}$\emph{\ (}left\emph{)}. Thus, the needed time
to complete this journey is close to $\left( \alpha _{\ell }\beta _{\ell
}-1\right) t_{0},$ which is the sharp time of observability $T_{\ell }$. 
\emph{Figure }$\ref{fig3}$\emph{\ (}right\emph{)} shows that we need the
same time $T_{\ell }$ for an initial disturbance concentrated near $x=-\ell
_{1}t_{0}$.
\end{remark}

\begin{figure}[tbph]
\centering
\includegraphics[width=120mm,height=50mm]{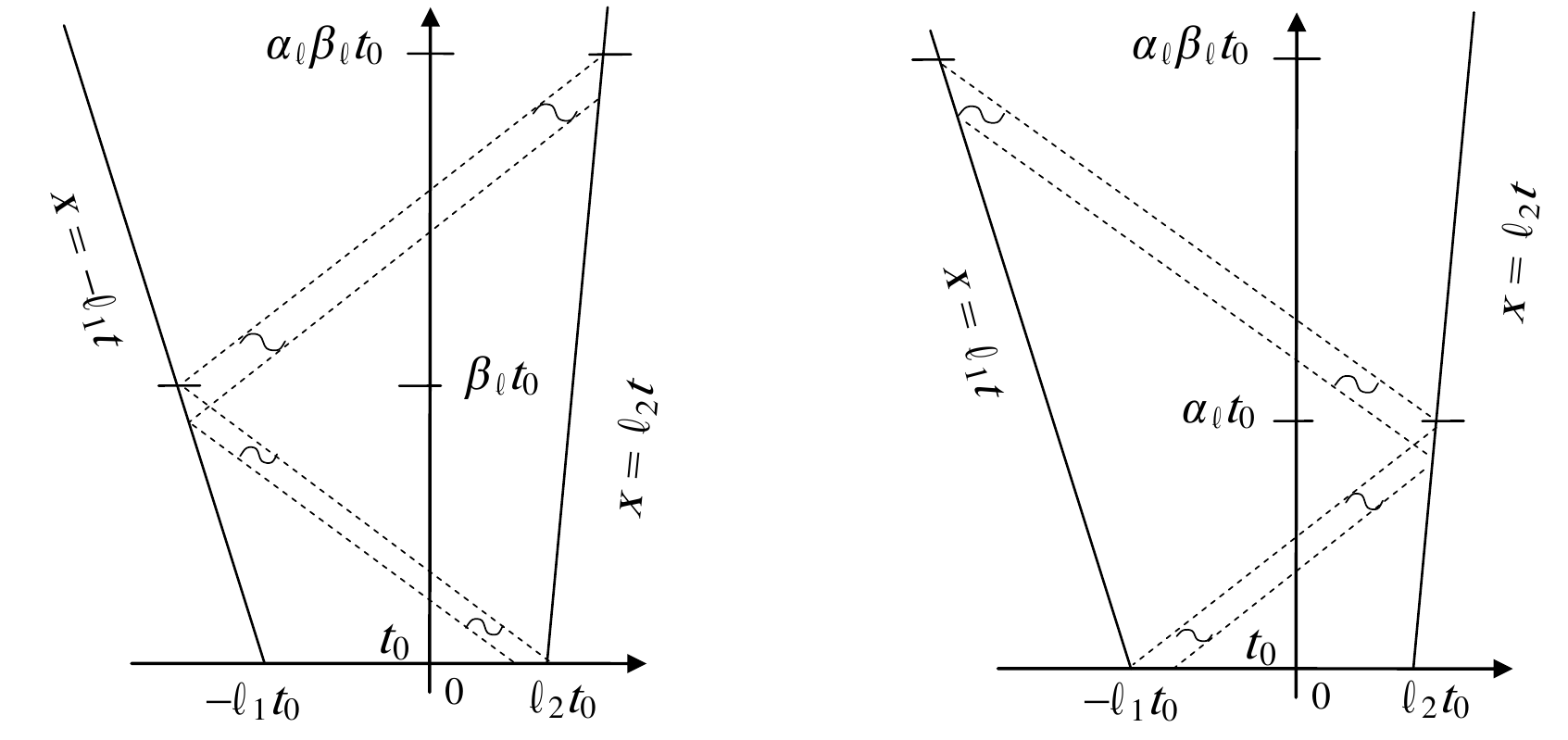}
\caption{Propagation of a wave with a small support near an endpoint $(\ell
_{2}<\ell _{1})$.}
\label{fig3}
\end{figure}

\subsection{Controllability at one endpoint}

First, let us check that the controllability problem (\ref{wavec}) can be
reduced to a null-controllability one, i.e., we can always assume that 
\begin{equation}
u\left( t_{0}+T\right) =u_{t}\left( t_{0}+T\right) =0\ \ \text{on }%
I_{t_{0}+T}.  \label{uT=0}
\end{equation}%
Indeed, consider the homogenous backward problem%
\begin{equation*}
\left\{ 
\begin{array}{ll}
z_{tt}-z_{xx}=0,\ \smallskip  & \text{in }Q_{t_{0}+T}, \\ 
u\left( -\ell _{1}t,t\right) =0,\text{ \ \ }z\left( \ell _{2}t,t\right)
=0,\smallskip  & \text{for \ \ }t\in \left( t_{0},t_{0}+T\right) , \\ 
z(x,t_{0}+T)=u_{T}^{0}\left( x\right) ,\text{ \ \ }z_{t}\left(
x,t_{0}+T\right) =u_{T}^{1}\left( x\right) ,\text{ \ } & \text{for \ \ }x\in
I_{t_{0}+T}.%
\end{array}%
\right. 
\end{equation*}%
One can argue as in \cite{Mir:96} to show that this problem has a solution
in the transposition sense. Then assume that there exists a function $v\in
L^{2}\left( t_{0},t_{0}+T\right) ,$ acting at the endpoint $x=\ell _{2}t,$
driving the solution of the forward problem 
\begin{equation*}
\left\{ 
\begin{array}{ll}
w_{tt}-w_{xx}=0,\ \smallskip  & \text{in }Q_{t_{0}+T}, \\ 
w\left( -\ell _{1}t,t\right) =0,\text{ \ \ }w\left( \ell _{2}t,t\right)
=v\left( t\right) ,\smallskip  & \text{for \ \ }t\in \left(
t_{0},t_{0}+T\right) , \\ 
w(x,t_{0})=u^{0}\left( x\right) -z(x,t_{0}),\text{ \ \ }w_{t}\left(
x,t_{0}\right) =u^{1}\left( x\right) -z_{t}(x,t_{0}),\text{ \ } & \text{for
\ \ }x\in I_{t_{0}}.%
\end{array}%
\right. 
\end{equation*}%
to the rest, i.e. $w(x,t_{0}+T)=w_{t}\left( x,t_{0}+T\right) =0.$ Then, $v$
drives $u=z+w$, solution of (\ref{wavec}), to $u(x,t_{0}+T)=u_{T}^{0}\left(
x\right) $ and $u_{t}\left( x,t_{0}+T\right) =u_{T}^{1}\left( x\right) $.

The null-controllability of (\ref{wavec}) at one of the endpoints is derived
by mean of HUM. Since the proofs are similar, only the case of the endpoint $%
x=\ell _{2}t$ is considered.

\begin{theorem}
\label{thc1}Under the assumptions \emph{(\ref{tlike})} and \emph{(\ref{u0}), 
}Problem \emph{(\ref{wavec}) }is exactly controllable at the endpoint $%
x=\ell _{2}t$ for $T\geq T_{\ell }.$ Moreover, we can choose a control $v$
satisfying 
\begin{equation}
\int_{t_{0}}^{t_{0}+T}v^{2}\left( t\right) dt\leq K_{\ell }\left( T\right)
E\left( t_{0}\right) ,  \label{vlt}
\end{equation}%
where $K_{\ell }\left( T\right) $ is a constant depending on $\ell _{1},\ell
_{2}$ and $T.$

Conversely, if $T<T_{\ell },$ Problem \emph{(\ref{wavec})} is not
controllable at $x=\ell _{2}t$.
\end{theorem}

\begin{proof}
$\bullet $ Let $\phi $ be the solution of problem (\ref{wave}). The idea of
HUM is to seek a control $v$ in the special form $v=\phi _{x}\left( \ell
_{2}t,t\right) \in L^{2}(t_{0},t_{0}+T),$ for a suitable choice of $\phi ^{0}
$ and $\phi ^{1}.$ First, we consider the backward problem%
\begin{equation}
\left\{ 
\begin{array}{ll}
\psi _{tt}-\psi _{xx}=0,\ \smallskip  & \text{in }Q_{t_{0}+T}, \\ 
\psi \left( -\ell _{1}t,t\right) =0,\text{ \ }\psi \left( \ell
_{2}t,t\right) =\phi _{x}\left( \ell _{2}t,t\right) ,\smallskip  & \ \text{%
for \ }t\in \left( t_{0},T\right) , \\ 
\psi \left( x,t_{0}+T\right) =0\text{,\ \ }\psi _{t}\left( x,t_{0}+T\right)
=0. & \ \text{for \ \ }x\in I_{t_{0}+T}.%
\end{array}%
\right.   \label{wave3}
\end{equation}%
We obtain the a linear map, that relates $(\phi ^{0},\phi ^{1})$ to the
initial data $(\psi _{t}\left( t_{0}\right) ,-\psi \left( t_{0}\right) )$ ,%
\begin{align*}
\Lambda _{1}:H_{0}^{1}\left( I_{t_{0}}\right) \times L^{2}\left(
I_{t_{0}}\right) & \longrightarrow H^{-1}\left( I_{t_{0}}\right) \times
L^{2}\left( I_{t_{0}}\right)  \\
(\phi ^{0},\phi ^{1})& \longmapsto (\psi _{t}\left( t_{0}\right) ,-\psi
\left( t_{0}\right) ).
\end{align*}%
The space $H_{0}^{1}\left( I_{t_{0}}\right) \times L^{2}\left(
I_{t_{0}}\right) $\ is equipped with the energy norm. To check that it is
possible to choose $(\phi ^{0},\phi ^{1})$ such that $(\psi _{t}\left(
t_{0}\right) ,-\psi \left( t_{0}\right) )=(u^{1},-u^{0})$, we argue as in 
\cite{Komo:94}. Since the solution of (\ref{wave3}) is taken in the
transposition sense, it comes that%
\begin{multline}
0=\int_{t_{0}}^{t_{0}+T}\left\langle \left( \psi _{tt}-\psi _{xx}\right)
,\phi \right\rangle _{H_{0}^{1}\left( I_{t}\right) }dt=-\left\langle \psi
_{t}\left( t_{0}\right) ,\phi ^{0}\right\rangle _{H_{0}^{1}\left(
I_{t_{0}}\right) }+\int_{-\ell _{1}t_{0}}^{\ell _{2}t_{0}}\psi \left(
t_{0}\right) \phi ^{1}dx  \label{diverg} \\
+\ell _{2}\int_{t_{0}}^{t_{0}+T}\psi \left( \ell _{2}t,t\right) \phi
_{t}\left( \ell _{2}t,t\right) dt+\int_{t_{0}}^{t_{0}+T}\psi \left( \ell
_{2}t,t\right) \phi _{x}\left( \ell _{2}t,t\right) dt
\end{multline}%
where $\left\langle \text{ },\text{ }\right\rangle _{X}$ denotes the duality
product between a Banach space $X$ and its dual. Observing that the boundary
condition $\phi \left( \ell _{2}t,t\right) =0$ implies that\ $\ell _{2}\phi
_{x}\left( \ell _{2}t,t\right) +\phi _{t}\left( \ell _{2}t,t\right) =0,$
i.e., 
\begin{equation*}
\phi _{t}\left( \ell _{2}t,t\right) =-\ell _{2}\phi _{x}\left( \ell
_{2}t,t\right) .
\end{equation*}%
Then, we can rewrite (\ref{diverg}) as%
\begin{equation*}
0=-\left\langle \psi _{t}\left( t_{0}\right) ,\phi ^{0}\right\rangle
_{H_{0}^{1}\left( I_{t_{0}}\right) }+\int_{-\ell _{1}t_{0}}^{\ell
_{2}t_{0}}\psi \left( t_{0}\right) \phi ^{1}dx-\ell
_{2}^{2}\int_{t_{0}}^{t_{0}+T}\phi _{x}^{2}\left( \ell _{2}t,t\right)
dt+\int_{t_{0}}^{t_{0}+T}\phi _{x}^{2}\left( \ell _{2}t,t\right) dt,
\end{equation*}%
hence 
\begin{eqnarray*}
\langle \Lambda _{1}(\phi ^{0},\phi ^{1}),(\phi ^{0},\phi ^{1})\rangle
_{H_{0}^{1}\left( I_{t_{0}}\right) \times L^{2}\left( I_{t_{0}}\right) }
&=&\left\langle \psi _{t}\left( t_{0}\right) ,\phi ^{0}\right\rangle
_{H_{0}^{1}\left( I_{t_{0}}\right) }-\int_{-\ell _{1}t_{0}}^{\ell
_{2}t_{0}}\psi \left( t_{0}\right) \phi ^{1}dx \\
&=&\left( 1-\ell _{2}^{2}\right) \int_{t_{0}}^{t_{0}+T}\phi _{x}^{2}\left(
\ell _{2}t,t\right) dt.
\end{eqnarray*}%
Thanks to Theorem \ref{thobs1}, we deduce that 
\begin{equation*}
\frac{4\left( 1-L\right) }{\alpha _{\ell }\beta _{\ell }\left( 1-\ell
_{2}^{2}\right) }E\left( t_{0}\right) \leq \langle \Lambda _{1}(\phi
^{0},\phi ^{1}),(\phi ^{0},\phi ^{1})\rangle \leq \left( 1-\ell
_{2}^{2}\right) K_{\ell }\left( T\right) \text{ }E\left( t_{0}\right) ,\text{
\ \ for }T\geq T_{\ell }.
\end{equation*}%
This means that $\Lambda _{1}$ is an isomorphism for $T\geq T_{\ell }$ and
therefore $(\phi ^{0},\phi ^{1})$ can be determined such that the control $%
v=\phi _{x}\left( \ell t,t\right) $ drive the solution of (\ref{wavec}) from
the initial data $u^{0},u^{1}$ to the rest, i.e. $u\left( t_{0}+T\right)
=u_{t}\left( t_{0}+T\right) =0$.

$\bullet $ If $T<T_{\ell },$ then (\ref{wave}) is not observable by Theorem %
\ref{thobs1}. This means that we can find non-zero initial data $(\phi
^{0},\phi ^{1})\in H_{0}^{1}\left( I_{t_{0}}\right) \times L^{2}\left(
I_{t_{0}}\right) $ such that 
\begin{equation}
\phi _{x}\left( \ell _{2}t,t\right) =0,\text{ \ \ }\forall t\in
(t_{0},t_{0}+T).  \label{v=0}
\end{equation}%
Choose $(u^{1},u^{0})\in H^{-1}\left( I_{t_{0}}\right) \times L^{2}\left(
I_{t_{0}}\right) $ such that%
\begin{equation*}
\left\langle u^{1},\phi ^{0}\right\rangle _{H_{0}^{1}\left( I_{t_{0}}\right)
}-\int_{-\ell _{1}t_{0}}^{\ell _{2}t_{0}}u^{0}\phi ^{1}dx\neq 0.
\end{equation*}%
Then, the solution of (\ref{wavec}), in the transposition sense, satisfies%
\begin{multline*}
\left\langle u_{t}\left( t_{0}+T\right) ,\phi \left( t_{0}+T\right)
\right\rangle _{H_{0}^{1}\left( I_{t_{0}+T}\right) }-\int_{-\ell _{1}\left(
t_{0}+T\right) }^{\ell _{2}\left( t_{0}+T\right) }u\left( t_{0}+T\right)
\phi _{t}\left( t_{0}+T\right) dx \\
-\left\langle u^{1},\phi ^{0}\right\rangle -\int_{-\ell _{1}t_{0}}^{\ell
_{2}t_{0}}u^{0}\phi ^{1}dx+\left( 1-\ell _{2}^{2}\right)
\int_{t_{0}}^{t_{0}+T}v\left( t\right) \phi _{x}\left( \ell _{2}t,t\right)
dt=0.
\end{multline*}%
Whatever the choice of the control function $v\in L^{2}(t_{0},t_{0}+T),$ the
last integral term always vanishes due to (\ref{v=0}). Hence 
\begin{multline*}
\left\langle u_{t}\left( t_{0}+T\right) ,\phi \left( t_{0}+T\right)
\right\rangle _{H_{0}^{1}\left( I_{t_{0}+T}\right) }-\int_{-\ell _{1}\left(
t_{0}+T\right) }^{\ell _{2}\left( t_{0}+T\right) }u\left( t_{0}+T\right)
\phi _{t}\left( t_{0}+T\right) dx \\
=\left\langle u^{1},\phi ^{0}\right\rangle _{H_{0}^{1}\left(
I_{t_{0}}\right) }-\int_{-\ell _{1}t_{0}}^{\ell _{2}t_{0}}u^{0}\phi
^{1}dx\neq 0
\end{multline*}%
and therefore we cannot have $u_{t}(t_{0}+T)=u(t_{0}+T)=0$ on $I_{t_{0}+T}$.
This completes the proof.
\end{proof}

\begin{remark}
The proof of Theorem \emph{\ref{thc1}} shows that the controllability of 
\emph{(\ref{wavec}), }at the endpoint $x=\ell _{2}t,$\emph{\ }is equivalent
to the observability of \emph{(\ref{wavec}) }at the same endpoint.
\end{remark}

\section{Observability and controllability at both endpoints}

If we can observe simultaneously the two endpoints of the interval, one
expects a shorter time of observability. The proof is more challenging in
this case.

\subsection{Observability at both endpoints}

Let us start by showing the following lemma.

\begin{lemma}
\label{lm5}Under the assumptions \emph{(\ref{tlike})} and \emph{(\ref{ic}), }%
the solution of \emph{(\ref{wave})} satisfies 
\begin{equation}
\left( 1-\ell _{1}^{2}\right) ^{2}\int_{t_{0}}^{\beta _{\ell
}^{M}t_{0}}t\phi _{x}^{2}(-\ell _{1}t,t)dt+\left( 1-\ell _{2}^{2}\right)
^{2}\int_{t_{0}}^{\alpha _{\ell }^{M}t_{0}}t\phi _{x}^{2}(\ell
_{2}t,t)dt=4MS_{\ell }  \label{slm2}
\end{equation}%
and it holds that%
\begin{multline}
4M\left( 1-L\right) t_{0}E\left( t_{0}\right) \leq \left( 1-\ell
_{1}^{2}\right) ^{2}\int_{t_{0}}^{\beta _{\ell }t_{0}}t\phi _{x}^{2}(-\ell
_{1}t,t)dt+\left( 1-\ell _{2}^{2}\right) ^{2}\int_{t_{0}}^{\alpha _{\ell
}t_{0}}t\phi _{x}^{2}(\ell _{2}t,t)dt  \label{EBlt2} \\
\leq 4M\left( 1+L\right) t_{0}E\left( t_{0}\right) .
\end{multline}
\end{lemma}

\begin{proof}
First, we establish (\ref{slm2}) for smooth initial data. Assume that $\phi
_{x}^{0}$ and $\phi ^{1}$ are continuous functions. This ensures in
particular that their generalized Fourier series are absolutely converging,
see \cite{BiR:89,Pin:09}. More precisely, the coefficients $c_{n},$ given by
(\ref{cn+}), satisfy 
\begin{equation}
\sum_{n\in 
\mathbb{Z}
^{\ast }}\left\vert nc_{n}\right\vert <+\infty .  \label{abso}
\end{equation}

On one hand, taking $x=-\ell _{1}t$ in (\ref{ph x}), we get%
\begin{eqnarray*}
\phi _{x}(-\ell _{1}t,t) &=&\pi \kappa _{\ell }\sum_{n\in 
\mathbb{Z}
^{\ast }}inc_{n}\left( \dfrac{e^{in\pi \kappa _{\ell }\log \left( 1-\ell
_{1}\right) t}}{\left( 1-\ell _{1}\right) t}+\dfrac{e^{in\pi \kappa _{\ell
}\log \left( \frac{1+\ell _{2}}{1-\ell _{2}}\left( 1+\ell _{1}\right)
t\right) }}{\left( 1+\ell _{1}\right) t}\right) \\
&=&\pi \kappa _{\ell }\sum_{n\in 
\mathbb{Z}
^{\ast }}inc_{n}\left( \dfrac{1}{1-\ell _{1}}+\dfrac{e^{in\pi \kappa _{\ell
}\log \alpha _{\ell }\beta _{\ell }}}{1+\ell _{1}}\right) \frac{e^{in\pi
\kappa _{\ell }\log \left( 1-\ell _{1}\right) t}}{t},
\end{eqnarray*}%
hence%
\begin{equation}
\phi _{x}\left( -\ell _{1}t,t\right) =\frac{2\pi \kappa _{\ell }}{1-\ell
_{1}^{2}}\sum_{n\in 
\mathbb{Z}
^{\ast }}inc_{n}\frac{e^{in\pi \kappa _{\ell }\log \left( 1-\ell _{1}\right)
t}}{t}.  \label{tphx-}
\end{equation}%
Let $m\in 
\mathbb{Z}
^{\ast }.$ Multiplying (\ref{tphx-}) by $\left( 1-\ell _{1}^{2}\right) $ $%
\overline{imc_{m}e^{im\pi \kappa _{\ell }\log \left( 1-\ell _{1}\right)
^{M}t}}$ and integrating on $\left( t_{0},\beta _{\ell }^{M}t_{0}\right) ,$
we get 
\begin{multline*}
\left( 1-\ell _{1}^{2}\right) \int_{t_{0}}^{\beta _{\ell }^{M}t_{0}}\phi
_{x}\left( -\ell _{1}t,t\right) \text{ }\overline{imc_{m}e^{im\pi \kappa
_{\ell }\log \left( 1-\ell _{1}\right) ^{M}t}}dt \\
=2\pi \kappa _{\ell }m\bar{c}_{m}\int_{t_{0}}^{\beta _{\ell
}^{M}t_{0}}\left( \sum_{n\in 
\mathbb{Z}
^{\ast }}nc_{n}e^{i\left( n-m\right) \pi \kappa _{\ell }\log \left( 1-\ell
_{1}\right) ^{M}t}\right) \frac{dt}{t}.
\end{multline*}%
Since $\left\vert nc_{n}e^{i\left( n-m\right) \pi \kappa _{\ell }\log \left(
1-\ell _{1}\right) ^{M}t}\right\vert =\left\vert nc_{n}\right\vert ,$ then
due to (\ref{abso}) the series in the right hand side is absolutely
converging. Applying Lebesgue's dominated convergence theorem and
integrating term-by-term, we obtain%
\begin{multline}
\left( 1-\ell _{1}^{2}\right) \int_{t_{0}}^{\beta _{\ell }^{M}t_{0}}\phi
_{x}\left( -\ell _{1}t,t\right) \text{ }\overline{imc_{m}e^{im\pi \kappa
_{\ell }\log \left( 1-\ell _{1}\right) ^{M}t}}dt  \label{A} \\
=2\pi \kappa _{\ell }\sum_{n\in 
\mathbb{Z}
^{\ast }}nmc_{n}\bar{c}_{m}e^{i\left( n-m\right) \pi \kappa _{\ell }\log
\left( 1-\ell _{1}\right) ^{M}}\int_{t_{0}}^{\beta _{\ell
}^{M}t_{0}}e^{i\left( n-m\right) \pi \kappa _{\ell }\log t}\frac{dt}{t}%
=\sum_{n\in 
\mathbb{Z}
^{\ast }}A_{nm}
\end{multline}%
where%
\begin{gather*}
A_{mm}=2M\pi \kappa _{\ell }\left\vert mc_{m}\right\vert ^{2}\log \beta
_{\ell }\text{,} \\
A_{nm}=\frac{2nmc_{n}\bar{c}_{m}}{i\left( n-m\right) }\left( e^{i\left(
n-m\right) M\pi \kappa _{\ell }\log \beta _{\ell }}-1\right) e^{i\left(
n-m\right) M\pi \kappa _{\ell }\log \left( 1-\ell _{1}\right) t_{0}}\text{ \
if }n\neq m.
\end{gather*}

On the other hand, for $x=\ell _{2}t,$ we multiply (\ref{tphx+}) by $\left(
1-\ell _{2}^{2}\right) $ $\overline{imc_{n}e^{im\pi \kappa _{\ell }\log
\left( 1+\ell _{2}\right) ^{M}t}},$ and integrate term-by-term on $\left(
t_{0},\alpha _{\ell }^{M}t_{0}\right) $, we end up with%
\begin{multline}
\left( 1-\ell _{2}^{2}\right) \int_{t_{0}}^{\alpha ^{M}t_{0}}\phi _{x}(\ell
_{2}t,t)\text{ }\overline{imc_{m}e^{im\pi \kappa _{\ell }\log \left( 1+\ell
_{2}\right) ^{M}t}}dt  \label{B} \\
=2\pi \kappa _{\ell }\sum_{n\in 
\mathbb{Z}
^{\ast }}nmc_{n}\bar{c}_{m}e^{i\left( n-m\right) \pi \kappa _{\ell }\log
\left( 1+\ell _{2}\right) ^{M}}\int_{t_{0}}^{\alpha ^{M}t_{0}}e^{i\left(
n-m\right) \pi \kappa _{\ell }\log t}\frac{dt}{t}=\sum_{n\in 
\mathbb{Z}
^{\ast }}B_{nm}
\end{multline}%
where%
\begin{gather*}
B_{mm}=2M\pi \kappa _{\ell }\left\vert mc_{m}\right\vert ^{2}\log \alpha
_{\ell }\text{,} \\
B_{nm}=\frac{2nmc_{n}\bar{c}_{m}}{i\left( n-m\right) }\left( e^{i\left(
n-m\right) M\pi \kappa _{\ell }\log \alpha _{\ell }}-1\right) e^{i\left(
n-m\right) M\pi \kappa _{\ell }\log \left( 1+\ell _{2}\right) t_{0}}\text{ \
if }n\neq m.
\end{gather*}%
Summing up (\ref{A}) and (\ref{B}), we obtain%
\begin{multline}
\left( 1-\ell _{1}^{2}\right) \int_{t_{0}}^{\beta _{\ell }^{M}t_{0}}\phi
_{x}\left( -\ell _{1}t,t\right) \text{ }\overline{imc_{n}e^{imM\pi \kappa
_{\ell }\log \left( 1-\ell _{1}\right) t}}dt  \label{A+B} \\
+\left( 1-\ell _{2}^{2}\right) \int_{t_{0}}^{\alpha _{\ell }^{M}t_{0}}\phi
_{x}(\ell _{2}t,t)\text{ }\overline{imc_{n}e^{imM\pi \kappa _{\ell }\log
\left( 1+\ell _{2}\right) t}}dt=\sum_{n\in 
\mathbb{Z}
^{\ast }}\left( A_{nm}+B_{nm}\right) .
\end{multline}%
Since $\kappa _{\ell }=2/\log \alpha _{\ell }\beta _{\ell },$ it comes that%
\begin{equation*}
A_{mm}+B_{mm}=2M\pi \kappa _{\ell }\left\vert mc_{m}\right\vert ^{2}\left(
\log \alpha _{\ell }+\log \beta _{\ell }\right) =4M\pi \left\vert
mc_{m}\right\vert ^{2}\text{, \ \ }m\in 
\mathbb{Z}
^{\ast }.
\end{equation*}%
If $n\neq m,$ then 
\begin{align*}
A_{nm}+B_{nm}=& \frac{2mnc_{n}\bar{c}_{m}}{i\left( n-m\right) }\left(
e^{i\left( n-m\right) M\pi \kappa _{\ell }\log \beta _{\ell }}-1\right)
e^{i\left( n-m\right) M\pi \kappa _{\ell }\log \left( 1-\ell _{1}\right)
t_{0}} \\
& +\frac{2mnc_{n}\bar{c}_{m}}{i\left( n-m\right) }\left( e^{i\left(
n-m\right) M\pi \kappa _{\ell }\log \alpha _{\ell }}-1\right) e^{i\left(
n-m\right) M\pi \kappa _{\ell }\log \left( 1+\ell _{2}\right) t_{0}} \\
=& \frac{2mnc_{n}\bar{c}_{m}\left( 1-e^{i\left( n-m\right) M\pi \kappa
_{\ell }\log \alpha _{\ell }}\right) }{i\left( n-m\right) }\left( e^{i\left(
n-m\right) M\pi \kappa _{\ell }\log \left( \frac{1-\ell _{1}}{\alpha _{\ell }%
}t_{0}\right) }-e^{i\left( n-m\right) M\pi \kappa _{\ell }\log \left( 1+\ell
_{2}\right) t_{0}}\right) \\
=& \frac{2mnc_{n}\bar{c}_{m}\left( 1-e^{i\left( n-m\right) M\pi \kappa
_{\ell }\log \alpha _{\ell }}\right) }{i\left( n-m\right) }\text{ }%
e^{i\left( n-m\right) M\pi \kappa _{\ell }\log \left( 1+\ell _{2}\right)
t_{0}}\left( e^{-i\left( n-m\right) M\pi \kappa _{\ell }\log \alpha _{\ell
}\beta _{\ell }}-1\right) .
\end{align*}%
The last parentheses vanishes, hence%
\begin{equation*}
A_{nm}+B_{nm}=0\text{ \ \ if }n\neq m,\text{ \ \ }n,m\in 
\mathbb{Z}
^{\ast }.
\end{equation*}%
Thus, we can rewrite (\ref{A+B}) as 
\begin{multline*}
\left( 1-\ell _{1}^{2}\right) \int_{t_{0}}^{\beta _{\ell }t_{0}}\phi
_{x}\left( -\ell _{1}t,t\right) \text{ }\overline{imc_{n}e^{im\pi \kappa
_{\ell }\log \left( 1-\ell _{1}\right) t}}dt \\
+\left( 1-\ell _{2}^{2}\right) \int_{t_{0}}^{\alpha _{\ell }t_{0}}\phi
_{x}(\ell _{2}t,t)\text{ }\overline{imc_{n}e^{im\pi \kappa _{\ell }\log
\left( 1+\ell _{2}\right) t}}dt=4M\pi \left\vert mc_{m}\right\vert ^{2},%
\text{ \ }m\in 
\mathbb{Z}
^{\ast }.
\end{multline*}%
Summing up for $m\in 
\mathbb{Z}
^{\ast }$, and applying Lebesgue's theorem to interchange summation and
integration, it comes that%
\begin{multline*}
\left( 1-\ell _{1}^{2}\right) \int_{t_{0}}^{\beta _{\ell }t_{0}}\phi
_{x}\left( -\ell _{1}t,t\right) \left( \sum\limits_{m=-\infty }^{+\infty }%
\overline{imc_{n}e^{im\pi \kappa _{\ell }\log \left( 1-\ell _{1}\right) t}}%
\right) dt \\
+\left( 1-\ell _{2}^{2}\right) \int_{t_{0}}^{\alpha _{\ell }t_{0}}\phi
_{x}(\ell _{2}t,t)\left( \sum\limits_{m=-\infty }^{+\infty }\overline{%
imc_{n}e^{im\pi \kappa _{\ell }\log \left( 1+\ell _{2}\right) t}}\right)
dt=4M\pi \sum_{m\in 
\mathbb{Z}
^{\ast }}\left\vert mc_{m}\right\vert ^{2}.
\end{multline*}%
Thanks to (\ref{tphx+}) and (\ref{tphx-}), we obtain%
\begin{equation*}
\left( 1-\ell _{1}^{2}\right) \int_{t_{0}}^{\beta _{\ell }t_{0}}\phi
_{x}^{2}\left( -\ell _{1}t,t\right) \left( t\frac{1-\ell _{1}^{2}}{2\pi
\kappa _{\ell }}\right) dt+\left( 1-\ell _{2}^{2}\right)
\int_{t_{0}}^{\alpha _{\ell }t_{0}}\phi _{x}^{2}(\ell _{2}t,t)\left( t\frac{%
1-\ell _{2}^{2}}{2\pi \kappa _{\ell }}\right) dt=4\pi \sum_{m\in 
\mathbb{Z}
^{\ast }}\left\vert mc_{m}\right\vert ^{2}
\end{equation*}%
After some rearrangements, (\ref{slm2}) follows as claimed.

In the general case, i.e. $\phi ^{0}\in H_{0}^{1}\left( -\ell _{1}t,\ell
_{2}t\right) $ and $\phi ^{1}\in L^{2}\left( -\ell _{1}t_{0},\ell
_{2}t_{0}\right) ,$ we use an argument of density. Consider two sequences $%
\phi _{j}^{0}\in C^{1}\left( \left[ -\ell _{1}t_{0},\ell _{2}t_{0}\right]
\right) $ and $\phi _{j}^{1}\in C\left( \left[ -\ell _{1}t_{0},\ell _{2}t_{0}%
\right] \right) ,j\in 
\mathbb{N}
,$ such that 
\begin{equation*}
\left( \phi _{j}^{0}\right) _{x}\rightarrow \phi _{x}^{0}\text{ \ and \ }%
\phi _{j}^{1}\rightarrow \phi ^{1}\text{ in }L^{2}\left( -\ell _{1}t,\ell
_{2}t\right) ,\text{ \ \ as }j\rightarrow +\infty ,
\end{equation*}%
then denote by $\phi _{j}$ and $E_{j}$ the solution and energy associated to
each data $\phi _{j}^{0},\phi _{j}^{1}$. Taking into account (\ref{est0})
for $t=t_{0}$, and using the precedent step of the proof, we have 
\begin{equation*}
\left( 1-\ell _{1}^{2}\right) ^{2}\int_{t_{0}}^{\beta _{\ell }t_{0}}t\left(
\phi _{j}\right) _{x}^{2}(-\ell _{1}t,t)dt+\left( 1-\ell _{2}^{2}\right)
^{2}\int_{t_{0}}^{\alpha _{\ell }t_{0}}t\left( \phi _{j}\right)
_{x}^{2}(\ell _{2}t,t)dt=4t_{0}E_{j}\left( t_{0}\right) +4\int\limits_{-\ell
_{1}t_{0}}^{\ell _{2}t_{0}}x\left( \phi _{j}^{0}\right) _{x}\phi _{j}^{1}\
dx.
\end{equation*}%
Relaying on the continuity of the solution of the wave equation with respect
to the initial data, which is emphasized by (\ref{stab}), the last
inequality holds as $j\rightarrow +\infty .$ This shows (\ref{slm2}) for the
general case.

The estimate (\ref{EBlt2}) follows by using (\ref{ESES}) for $t=t_{0}$.
\end{proof}

Let us now establish the observability of the wave equation (\ref{wave})%
\emph{\ }at the two endpoints $x=-\ell _{1}t$ and $x=\ell _{2}t$.

\begin{theorem}
\label{thobs2}Under the assumption \emph{(\ref{tlike}) }and \emph{(\ref{ic}),%
} we have:

$\bullet $ For every $T\geq 0,$ the solution of \emph{(\ref{wave})}
satisfies the direct inequality 
\begin{equation}
\int_{t_{0}}^{t_{0}+T}\phi _{x}^{2}(-\ell _{1}t,t)+\phi _{x}^{2}(\ell
_{2}t,t)dt\leq \tilde{K}_{\ell }\left( T\right) E\left( t_{0}\right) .
\label{direct2}
\end{equation}%
with a constant $\tilde{K}_{\ell }\left( T\right) $ depending only on $\ell
_{1},\ell _{2}$ and $T.$

$\bullet $ If $T\geq \tilde{T}_{\ell }=\left( \max \left\{ \alpha _{\ell
},\beta _{\ell }\right\} -1\right) t_{0},$ Problem \emph{(\ref{wave})} is
observable at the two endpoints $x=-\ell _{1}t,x=\ell _{2}t$, and it holds
that%
\begin{equation}
E\left( t_{0}\right) \leq \frac{\left( 1-l^{2}\right) ^{2}\max \left\{
\alpha _{\ell },\beta _{\ell }\right\} }{4\left( 1-L\right) }%
\int_{t_{0}}^{t_{0}+T}\phi _{x}^{2}(-\ell _{1}t,t)+\phi _{x}^{2}(\ell
_{2}t,t)dt.  \label{obs2}
\end{equation}%
Conversely, if $T<\tilde{T}_{\ell },$ \emph{(\ref{obs2})} does not hold.
\end{theorem}

\begin{proof}
$\bullet $ The right-hand side of Inequality (\ref{EBlt2}) yields%
\begin{equation*}
\left( 1-L^{2}\right) ^{2}t_{0}\left( \int_{t_{0}}^{\beta _{\ell
}^{M}t_{0}}\phi _{x}^{2}(-\ell _{1}t,t)dt+\int_{t_{0}}^{\alpha _{\ell
}^{M}t_{0}}\phi _{x}^{2}(\ell _{2}t,t)dt\right) \leq 4M\left( 1+L\right)
t_{0}\text{ }E\left( t_{0}\right) ,
\end{equation*}%
and thus%
\begin{equation}
\int_{t_{0}}^{\beta _{\ell }^{M}t_{0}}\phi _{x}^{2}(-\ell
_{1}t,t)dt+\int_{t_{0}}^{\alpha _{\ell }^{M}t_{0}}\phi _{x}^{2}(\ell
_{2}t,t)dt\leq \frac{4M\left( 1+L\right) }{\left( 1-L^{2}\right) ^{2}}\
E\left( t_{0}\right) .  \label{EB0nd+}
\end{equation}%
Since $\min \left\{ \alpha _{\ell },\beta _{\ell }\right\} >1,$ we can
always choose the integer $M$ such that%
\begin{equation*}
\left( \min \left\{ \alpha _{\ell },\beta _{\ell }\right\} \right)
^{M}t_{0}\geq t_{0}+T.
\end{equation*}%
As the integrated functions in (\ref{EB0nd+}) are nonnegative, then (\ref%
{direct2}) holds with 
\begin{equation*}
\tilde{K}_{\ell }\left( T\right) :=4M\left( 1+L\right) /\left(
1-L^{2}\right) ^{2}.
\end{equation*}

$\bullet $ The right-hand side of (\ref{EBlt2}), for $M=1,$ yields%
\begin{equation*}
4\left( 1-L\right) t_{0}E\left( t_{0}\right) \leq \left( 1-l^{2}\right)
^{2}\max \left\{ \alpha _{\ell },\beta _{\ell }\right\}
t_{0}\int_{t_{0}}^{\max \left\{ \alpha _{\ell },\beta _{\ell }\right\}
t_{0}}\phi _{x}^{2}(-\ell _{1}t,t)+\phi _{x}^{2}(\ell _{2}t,t)dt
\end{equation*}%
and thus inequality (\ref{obs2}) holds for $T=\max \left\{ \alpha _{\ell
},\beta _{\ell }\right\} t_{0}-t_{0}=\tilde{T}_{\ell }$ and therefore for
every $T\geq \tilde{T}_{\ell }$ as well.

To show that (\ref{obs1}) does not hold for $T<\tilde{T}_{\ell },$ let us
assume that $\ell _{2}\geq \ell _{1}.$ The other case can be treated
similarly. We have then $\max \left\{ \alpha _{\ell },\beta _{\ell }\right\}
=\alpha _{\ell }$ and $\tilde{T}_{\ell }=\left( \alpha _{\ell }-1\right)
t_{0}.$ Let $\varepsilon >0,\varepsilon \in \left( 0,1\right) ,$ and
consider again a function $g\in L^{2}\left( t_{0},\alpha _{\ell }\beta
_{\ell }t_{0},dt/t\right) ,$ non identical null, satisfying (\ref{intgg})
and (\ref{gseries}) with a support satisfying this time%
\begin{equation}
\text{supp}\left( g\right) \subset \subset \left( \left( 1-\varepsilon
\right) \alpha _{\ell }t_{0},\alpha _{\ell }t_{0}\right) .  \label{h}
\end{equation}%
The $\mathbf{c}_{n}$ are defined as in (\ref{cng}) and $\tilde{\phi}_{x}^{0},%
\tilde{\phi}^{1}$ are chosen to satisfy (\ref{icchoice}) on $\left( -\ell
_{1}t_{0},L_{2}t_{0}\right) $. The solution of (\ref{wave}) corresponding to
these initial conditions still satisfies 
\begin{equation}
\phi _{x}(\ell _{2}t,t)=\frac{g\left( t\right) }{\left( 1-\ell
_{2}^{2}\right) t}.  \label{gl2}
\end{equation}%
Taking $M=1$ in (\ref{tphx-}), we get 
\begin{multline*}
\phi _{x}\left( -\ell _{1}t,t\right) =\frac{2\pi \kappa _{\ell }}{\left(
1-\ell _{1}^{2}\right) t}\sum_{n\in 
\mathbb{Z}
^{\ast }}in\mathbf{c}_{n}e^{in\pi \kappa _{\ell }\log \left( 1-\ell
_{1}\right) }\ e^{in\pi \kappa _{\ell }\log t} \\
=\frac{2\pi \kappa _{\ell }}{\left( 1-\ell _{1}^{2}\right) t}\sum_{n\in 
\mathbb{Z}
^{\ast }}in\mathbf{c}_{n}e^{in\pi \kappa _{\ell }\log \left( 1+\ell
_{2}\right) }\ e^{in\pi \kappa _{\ell }\log \left( \alpha _{\ell }t\right) }=%
\frac{1}{\left( 1-\ell _{1}^{2}\right) t}\sum_{n\in 
\mathbb{Z}
^{\ast }}g_{n}\ e^{in\pi \kappa _{\ell }\log \left( \alpha _{\ell }t\right) }
\end{multline*}%
since $\left( 1-\ell _{1}\right) \alpha _{\ell }\beta _{\ell }=\left( 1+\ell
_{2}\right) \alpha _{\ell }$, hence%
\begin{equation}
\phi _{x}\left( -\ell _{1}t,t\right) =\frac{g\left( \alpha _{\ell }t\right) 
}{\left( 1-\ell _{1}^{2}\right) t}.  \label{gl1}
\end{equation}%
Taking the squares in (\ref{gl2}) and (\ref{gl1}), summing up then
integrating on $\left( t_{0},\left( 1-\varepsilon \right) \alpha _{\ell
}t_{0}\right) ,$ it comes that%
\begin{multline}
\int_{t_{0}}^{\left( 1-\varepsilon \right) \alpha _{\ell }t_{0}}\phi
_{x}^{2}(-\ell _{1}t,t)dt+\int_{t_{0}}^{\left( 1-\varepsilon \right) \alpha
_{\ell }t_{0}}\phi _{x}^{2}(\ell _{2}t,t)dt  \label{gl12} \\
=\frac{1}{\left( 1-\ell _{1}^{2}\right) ^{2}}\int_{t_{0}}^{\left(
1-\varepsilon \right) \alpha _{\ell }t_{0}}g^{2}\left( \alpha _{\ell
}t\right) \frac{dt}{t^{2}}+\frac{1}{\left( 1-\ell _{2}^{2}\right) ^{2}}%
\int_{t_{0}}^{\left( 1-\varepsilon \right) \alpha _{\ell }t_{0}}g^{2}\left(
t\right) \frac{dt}{t^{2}}.
\end{multline}%
Clearly, the last integral vanishes since supp$\left( g\right) \cap \left(
t_{0},\left( 1-\varepsilon \right) \alpha _{\ell }t_{0}\right) =\varnothing .
$ In addition, we have 
\begin{equation*}
\int_{t_{0}}^{\left( 1-\varepsilon \right) \alpha _{\ell }t_{0}}g^{2}\left(
\alpha _{\ell }t\right) \frac{dt}{t^{2}}=\alpha _{\ell }\int_{\alpha _{\ell
}t_{0}}^{\left( 1-\varepsilon \right) \alpha _{\ell }^{2}t_{0}}g^{2}\left(
s\right) \frac{ds}{s^{2}}.
\end{equation*}%
Since $g$, given by the series (\ref{gseries}), satisfies $g\left( \alpha
_{\ell }\beta _{\ell }s\right) =g\left( s\right) $ then it can also be
considered as a function in $L^{2}\left( \alpha _{\ell }t_{0},\alpha _{\ell
}^{2}\beta _{\ell }t_{0},\frac{ds}{s}\right) ,$ it suffices to take $%
a=\alpha _{\ell }t_{0},b=\alpha _{\ell }^{2}\beta _{\ell }t_{0}$ and $M=1$
in Lemma \ref{lm01}. In particular, we have necessarily%
\begin{equation*}
\text{supp}\left( g\right) \subset \subset \left( \left( 1-\varepsilon
\right) \alpha _{\ell }^{2}\beta _{\ell }t_{0},\alpha _{\ell }^{2}\beta
_{\ell }t_{0}\right) \text{, \ \ \ for }s\in \left( \alpha _{\ell
}t_{0},\alpha _{\ell }^{2}\beta _{\ell }t_{0}\right) .
\end{equation*}%
Noting that $\alpha _{\ell }^{2}t_{0}<\alpha _{\ell }^{2}\beta _{\ell }t_{0}$%
, then 
\begin{equation*}
\text{supp}\left( g\right) \cap \left( \alpha _{\ell }t_{0},\left(
1-\varepsilon \right) \alpha _{\ell }^{2}t_{0}\right) =\varnothing \text{, \
\ \ for }s\in \left( \alpha _{\ell }t_{0},\alpha _{\ell }^{2}\beta _{\ell
}t_{0}\right) 
\end{equation*}%
for $\varepsilon \in \left( 0,1\right) $ and by consequence 
\begin{equation*}
\int_{\alpha _{\ell }t_{0}}^{\left( 1-\varepsilon \right) \alpha _{\ell
}^{2}t_{0}}g^{2}\left( s\right) \frac{ds}{s^{2}}=0.
\end{equation*}%
Going back to (\ref{gl12}), we deduce that%
\begin{equation*}
\int_{t_{0}}^{\left( 1-\varepsilon \right) \alpha _{\ell }t_{0}}\phi
_{x}^{2}(-\ell _{1}t,t)+\phi _{x}^{2}(\ell _{2}t,t)dt=0,
\end{equation*}%
which means that the observability inequality (\ref{obs1}) does not hold for
every $T<\tilde{T}_{\ell }$.
\end{proof}

\subsection{Controllability at both endpoints}

Arguing as in the precedent section, it suffices to show the exact
null-controllability of Problem (\ref{wavec2}), with two controls acting at
the two endpoints.

\begin{theorem}
\label{thc2}Under the assumptions \emph{(\ref{tlike})} and \emph{(\ref{y0}),}
Problem \emph{(\ref{wavec2}) }is exactly controllable at the two endpoints $%
x=-\ell _{1}t,x=\ell _{2}t$ for $T\geq \tilde{T}_{\ell }.$ Moreover, we can
choose two controls $v_{1},v_{2}$ satisfying 
\begin{equation}
\int_{t_{0}}^{t_{0}+T}v_{1}^{2}\left( t\right) dt,\text{ \ }%
\int_{t_{0}}^{t_{0}+T}v_{2}^{2}\left( t\right) dt\leq \tilde{K}_{\ell
}\left( T\right) E\left( t_{0}\right) ,  \label{vlt2}
\end{equation}%
where $\tilde{K}_{\ell }\left( T\right) $ is a constant depending on $\ell
_{1},\ell _{2}$ and $T.$

Conversely, if $T<\tilde{T}_{\ell },$ Problem \emph{(\ref{wavec2})} is not
controllable at both endpoints $x=-\ell _{1}t$ and $x=\ell _{2}t$.
\end{theorem}

\begin{proof}
We argue as in the proof of Theorem \ref{thc1}. Let $\eta $ be the solution
of the backward problem%
\begin{equation}
\left\{ 
\begin{array}{ll}
\eta _{tt}-\eta _{xx}=0,\ \smallskip  & \text{in }Q_{t_{0}+T}, \\ 
\eta \left( -\ell _{1}t,t\right) =\phi _{x}\left( -\ell _{1}t,t\right) ,%
\text{ \ \ }\eta \left( \ell _{2}t,t\right) =\phi _{x}\left( \ell
_{2}t,t\right) ,\smallskip  & \text{for \ }t\in \left( t_{0},T\right) , \\ 
\eta \left( x,t_{0}+T\right) =0,\text{\ \ \ }\eta _{t}\left(
x,t_{0}+T\right) =0. & \text{for \ \ }x\in I_{t_{0}+T}.%
\end{array}%
\right.   \label{wave4}
\end{equation}%
We obtain then a linear map%
\begin{align*}
\Lambda _{2}:H_{0}^{1}\left( I_{t_{0}}\right) \times L^{2}\left(
I_{t_{0}}\right) & \longrightarrow H^{-1}\left( I_{t_{0}}\right) \times
L^{2}\left( I_{t_{0}}\right)  \\
(\phi ^{0},\phi ^{1})& \longmapsto (\eta _{t}\left( t_{0}\right) ,-\eta
\left( t_{0}\right) ).
\end{align*}%
The solution of (\ref{wave4}), in the transposition sense, satisfies%
\begin{multline}
-\left\langle \eta _{t}\left( t_{0}\right) ,\phi ^{0}\right\rangle
_{H_{0}^{1}\left( I_{t_{0}}\right) }+\int_{-\ell _{1}t_{0}}^{\ell
_{2}t_{0}}\eta \left( t_{0}\right) \phi ^{1}dx  \label{diverg2} \\
+\ell _{1}\int_{t_{0}}^{t_{0}+T}\eta \left( -\ell _{1}t,t\right) \phi
_{t}\left( -\ell _{1}t,t\right) dt-\int_{t_{0}}^{t_{0}+T}\eta \left( -\ell
_{1}t,t\right) \phi _{x}\left( -\ell _{1}t,t\right) dt \\
+\ell _{2}\int_{t_{0}}^{t_{0}+T}\eta \left( \ell _{2}t,t\right) \phi
_{t}\left( \ell _{2}t,t\right) dt+\int_{t_{0}}^{t_{0}+T}\eta \left( \ell
_{2}t,t\right) \phi _{x}\left( \ell _{2}t,t\right) dt=0
\end{multline}%
Taking into account that $\phi _{t}\left( -\ell _{1}t,t\right) =\ell
_{1}\phi _{x}\left( \ell _{2}t,t\right) \ $and $\phi _{t}\left( \ell
_{2}t,t\right) =-\ell _{2}\phi _{x}\left( \ell _{2}t,t\right) $, then we can
rewrite (\ref{diverg2}) as%
\begin{multline*}
0=-\left\langle \eta _{t}\left( t_{0}\right) ,\phi ^{0}\right\rangle
_{H_{0}^{1}\left( I_{t_{0}}\right) }+\int_{-\ell _{1}t_{0}}^{\ell
_{2}t_{0}}\eta \left( t_{0}\right) \phi ^{1}dx \\
+\left( 1-\ell _{1}^{2}\right) \int_{t_{0}}^{t_{0}+T}\phi _{x}^{2}\left(
-\ell _{1}t,t\right) dt+\left( 1-\ell _{2}^{2}\right)
\int_{t_{0}}^{t_{0}+T}\phi _{x}^{2}\left( \ell _{2}t,t\right) dt,
\end{multline*}%
i.e., 
\begin{eqnarray*}
\langle \Lambda _{2}(\phi ^{0},\phi ^{1}),(\phi ^{0},\phi ^{1})\rangle 
&=&\left\langle \eta _{t}\left( t_{0}\right) ,\phi ^{0}\right\rangle
_{H_{0}^{1}\left( I_{t_{0}}\right) }-\int_{-\ell _{1}t_{0}}^{\ell
_{2}t_{0}}\eta \left( t_{0}\right) \phi ^{1}dx \\
&=&\left( 1-\ell _{1}^{2}\right) \int_{t_{0}}^{t_{0}+T}\phi _{x}^{2}\left(
-\ell _{1}t,t\right) dt+\left( 1-\ell _{2}^{2}\right)
\int_{t_{0}}^{t_{0}+T}\phi _{x}^{2}\left( \ell _{2}t,t\right) dt.
\end{eqnarray*}%
Thanks to Theorem \ref{thobs2}, we have 
\begin{equation*}
\frac{4\left( 1-L\right) ^{2}}{\max \left\{ \alpha _{\ell },\beta _{\ell
}\right\} \left( 1-l^{2}\right) }E\left( t_{0}\right) \leq \langle \Lambda
_{2}(\phi ^{0},\phi ^{1}),(\phi ^{0},\phi ^{1})\rangle \leq \left(
1-l^{2}\right) \tilde{K}_{\ell }\left( T\right) \text{ }E\left( t_{0}\right)
,\text{ \ \ for }T\geq \tilde{T}_{\ell }.
\end{equation*}%
This means that $\Lambda _{2}$ is an isomorphism for $T\geq \tilde{T}_{\ell }
$ and thus $(\phi ^{0},\phi ^{1})$ can be determined such that the control $%
v_{1}=\phi _{x}\left( -\ell _{1}t,t\right) $ and $v_{2}=\phi _{x}\left( \ell
_{2}t,t\right) $ drive the solution of (\ref{wavec2}) from the initial data $%
y^{0},y^{1}$ to $y\left( t_{0}+T\right) =y_{t}\left( t_{0}+T\right) =0$.

$\bullet $ If $T<\tilde{T}_{\ell },$ then Problem (\ref{wave}) is not
observable. This means that we can find non-zero initial data $(\phi
^{0},\phi ^{1})\in H_{0}^{1}\left( I_{t_{0}}\right) \times L^{2}\left(
I_{t_{0}}\right) $ such that%
\begin{equation}
\phi _{x}\left( -\ell _{1}t,t\right) =\phi _{x}\left( \ell _{2}t,t\right) =0,%
\text{ \ \ }\forall t\in (t_{0},t_{0}+T).  \label{2v=0}
\end{equation}%
Choose $(y^{1},y^{0})\in H^{-1}\left( I_{t_{0}}\right) \times L^{2}\left(
I_{t_{0}}\right) $ such that%
\begin{equation*}
\left\langle y^{1},\phi ^{0}\right\rangle _{H_{0}^{1}\left( I_{t_{0}}\right)
}-\int_{-\ell _{1}t_{0}}^{\ell _{2}t_{0}}y^{0}\phi ^{1}dx\neq 0.
\end{equation*}%
Then, the solution of (\ref{wavec2}), in the transposition sense, satisfies%
\begin{multline*}
\left\langle y_{t}\left( t_{0}+T\right) ,\phi \left( t_{0}+T\right)
\right\rangle _{H_{0}^{1}\left( I_{t_{0}}\right) }-\int_{-\ell _{1}\left(
t_{0}+T\right) }^{\ell _{2}\left( t_{0}+T\right) }y\left( t_{0}+T\right)
\phi _{t}\left( t_{0}+T\right) dx-\left\langle y^{1},\phi ^{0}\right\rangle
_{H_{0}^{1}\left( I_{t_{0}}\right) } \\
+\int_{-\ell _{1}t_{0}}^{\ell _{2}t_{0}}y^{0}\phi ^{1}dx+\left( 1-\ell
_{1}^{2}\right) \int_{t_{0}}^{t_{0}+T}v_{1}\left( t\right) \phi _{x}\left(
-\ell _{1}t,t\right) dt+\left( 1-\ell _{2}^{2}\right)
\int_{t_{0}}^{t_{0}+T}v_{2}\left( t\right) \phi _{x}\left( \ell
_{2}t,t\right) dt=0.
\end{multline*}%
Whatever the choice of $v_{1},v_{2}\in L^{2}(t_{0},t_{0}+T),$ the third and
fourth integral terms vanish due to (\ref{2v=0}), hence we cannot have $%
y_{T}(t_{0}+T)=y(t_{0}+T)=0$ on $I_{t_{0}+T}$. This completes the proof.
\end{proof}

\begin{remark}
The observability and controllability of the 1-d wave equation in
noncylindrical domains, with other types of boundary conditions, are
established by the same approach used in this paper. The results will be
published elsewhere.
\end{remark}


\end{document}